\documentclass[11pt,reqno]{amsart}
\usepackage{amssymb,version,graphicx,fancybox,pifont,color}
\usepackage{url,hyperref}
\usepackage{subeqnarray}
\usepackage{leftidx}
\usepackage{float}
\usepackage{amsmath}
\usepackage{bm, bbm}
\usepackage{caption}
\usepackage{diagbox}
\usepackage{subcaption}
\usepackage{epstopdf}
\usepackage{mathrsfs}
\usepackage{graphicx}
\usepackage{soul}

\allowdisplaybreaks
\numberwithin{equation}{section}
\newtheorem{thm}{Theorem}
\newtheorem{remark}{Remark}

\usepackage{algorithm}
\usepackage{pifont}
\usepackage[noend]{algpseudocode}
\algnewcommand{\Inputs}[1]{%
  \State \textbf{Inputs:}
  \Statex \hspace*{\algorithmicindent}\parbox[t]{.8\linewidth}{\raggedright #1}
}

\begin{document}
\title[A SAV based algorithm  for discrete gradient systems]{An efficient and robust SAV based algorithm  for discrete gradient systems arising from optimizations}%
\author{Xinyu Liu${}^1$, Jie Shen${}^1$, and  Xiaongxiong Zhang${}^1$}
\thanks{${}^1$Department of Mathematics, Purdue University, West Lafayette, IN 47907, USA  (liu1957$@$purdue.edu, shen7$@$purdue.edu, zhan1966$@$purdue.edu).}

%
%
\begin{abstract}
	We propose in this paper a new minimization algorithm based on a slightly modified version of the scalar auxialiary variable (SAV)  approach  coupled with a relaxation  step and an adaptive strategy. It enjoys several distinct advantages over popular gradient based methods: (i) it is  unconditionally energy  diminishing with a  modified energy which is intrinsically related to the original energy, thus no parameter tuning is needed for stability; (ii)  it  allows the use of  large step-sizes which can effectively accelerate the convergence rate. We also present a convergence analysis for some SAV based algorithms, which include our new algorithm without the relaxation step as  a special case.
	We apply our new algorithm to several illustrative and benchmark problems, and compare its performance with several  popular gradient based methods. The numerical results indicate that the new algorithm is very robust, and  its adaptive version usually converges significantly faster than those popular gradient descent based methods.  
\end{abstract}
\keywords{ discrete gradient system; optimization; scalar auxiliary variable; adaptive; stability; convergence}
\subjclass[2020]{65K10,49M05,90C26}
 \maketitle
 \baselineskip=15pt
\section{Introduction}\label{sec: intro}

Minimization plays an important role in many branches of science and engineering. In particular,  how to accelerate the convergence rate of the minimization process is a central issue  in data science and machine learning problems.
We consider in this paper an unconstrained minimization problem 
\begin{equation}\label{P}\min\limits_{\theta\in \mathbbm R^N} f(\theta)
\end{equation}which  arises in many {applications}, including particularly machine learning problems. For large scale minimization problems, {the  first order methods such as gradient descent, its variants such as stochastic gradient descent} \cite{robbins1951stochastic}, Nesterov's accelerated gradient descent \cite{nesterov2013introductory}, adaptive momentum
estimation method \cite{kingma2014adam, sashank2018convergence, guo2021novel}, are popular choices. We refer to \cite{nocedal2006numerical, nesterov2018lectures, ryu2022large}, and the references therein,  for more detail on the  design and analysis of gradient descent method and its various variants.




The vanilla gradient descent method for  \eqref{P} is 
\begin{equation}
\label{Vanilla-GD}
\theta_{k+1}=\theta_{k}- \eta_k \nabla f(\theta_k),
\end{equation}
 which can also be regarded as a numerical scheme for the gradient flow  
 \begin{equation}
 	\label{G}
 	\theta_t=-\nabla f(\theta),
 \end{equation} with time step $ \eta_k$. The gradient flow \eqref{G} is energy diminishing in the sense that 
$$\frac d{dt} f(\theta)=(\nabla f(\theta), \theta_t)=-(\theta_t, \theta_t ) =- \|\theta_t\|^2  \le 0,$$
 where $(\cdot,\cdot)$ (resp. $\|\cdot\|$) denotes the $l^2$ inner product (resp. norm).
However, gradient decent type schemes are {not necessarily} energy diminishing, and may blow up if the time step is too large.  Although the stability of gradient descent based methods is well understood, the main challenge in practice  is how to choose the step-size, i.e. learning rate, to balance between stability and efficiency \cite{ruder2016overview}.

We propose in this paper a different class of minimization algorithms inspired from the recently developed scalar auxiliary variable (SAV) approach for gradient flows \cite{SXY18,shen2019new}. The SAV approach enjoys a particular advantage of unconditional energy diminishing compared to popular gradient decent based methods. This advantage avoids tuning step sizes and allows the use of large step sizes, which may effectively accelerate the convergence rate. 

Assume the cost function has a splitting
\begin{equation}\label{splitting}
f(\theta)= \frac{1}{2}(\mathcal{L}\theta,\theta)+[f(\theta) - \frac{1}{2}(\mathcal{L}\theta,\theta)]:= 
 \frac{1}{2}(\mathcal{L}\theta,\theta)+g(\theta),
 \end{equation} where  $\mathcal{L}$ is {a self-adjoint positive semi-definite linear operator}.  Note that $\mathcal{L}=0$ is a trivial splitting.
 Then,  the gradient flow  \eqref{G} becomes 
\begin{equation}\label{grad_flow2}
\theta_t + \mathcal{L} \theta  + \nabla g(\theta) = 0.
\end{equation}
Inspired by the IEQ and SAV approaches \cite{yang2016linear,SXY18}, assuming there exists $C>0$ such that $g(\theta)>-C$ for all $\theta$,
we introduce a scalar auxiliary variable $r(t) =  \sqrt{g(\theta)+C}$, and   expand \eqref{grad_flow2} to:
\begin{equation}\label{sav-general2}
    \begin{cases}
    \theta_t + \mathcal{L} \theta  + \frac{\nabla g(\theta)}{\sqrt{g(\theta)+C}}r = 0, \\
    r_t = \frac{1}{2\sqrt{g(\theta)+C}}( \nabla g(\theta),  \theta_t).
    \end{cases}
\end{equation}
Obviously, with $r(0)=  \sqrt{g(\theta|_{t=0}) + C}$, the above system admits a solution $r(t) =  \sqrt{g(\theta)+ C}$ with $\theta$ being the solution of \eqref{grad_flow2}. The main advantage of  the expanded system, which includes an energy evolution  equation, is that it allows us to construct simple numerical schemes with modified energy diminishing. For example, 
the following scheme
\begin{equation}\label{scheme-general2}
    \begin{cases}
    \frac{\theta_{k+1} - \theta_{k}}{\delta t} +  \mathcal{L}  \theta_{k+1} + \frac{\nabla g(\theta_{k})}{\sqrt{g(\theta_{k})+C}}r_{k+1}  = 0, \\
    \frac{r_{k+1} - r_{k}}{\delta t} = \left( \frac{\nabla g(\theta_{k})}{2\sqrt{g(\theta_{k})+C}},  \frac{\theta_{k+1} - \theta_{k}}{\delta t}\right),
    \end{cases}
\end{equation}
 can be easily implemented by solving only two linear systems of the form $(I + \delta t\mathcal{L}) x=b$, and is unconditionally energy stable with a modified energy \cite{shen2019new}. 
 
 While the  scheme \eqref{scheme-general2} has been shown to be very effective for gradient flows, it is not particularly suitable for the minimization problem \eqref{P}. Indeed, for any fixed $\delta t$, assuming $\theta_{k} \rightarrow  \theta_*$ and $r_{k}\rightarrow r_*$, then $ \frac{r_k}{\sqrt{g(\theta_{k})+C}}\rightarrow \frac{r_{*}}{\sqrt{g(\theta_{*})+C}} $ which is usually not equal to 1, and consequently, the first equation of  \eqref{scheme-general2} leads to $$0= \mathcal{L} \theta _*+\frac{r_{*}}{\sqrt{g(\theta_{*})+C}} 
 \nabla g(\theta_*)=\mathcal{L} \theta _*+\frac{r_{*}}{\sqrt{g(\theta_{*})+C}} 
 \left(-\mathcal{L} \theta _*+ \nabla f(\theta_*)\right). $$  
  If $\mathcal{L}\neq 0$, we observe from the above that    in general $ \nabla f(\theta_*)\ne 0$ thus $\theta_*$ is not a solution for  \eqref{P}. Another complication  of this approach is that it is not obvious how to choose $\mathcal{L}$ such that $g(\theta)$ is  bounded from below for all $\theta$.
The main goal of this paper is to design  suitable SAV  based schemes for \eqref{P}, develop their convergence theory, and validate them through extensive numerical experiments.

The rest of the paper is organized as follows. In Section \ref{sec-rsav}, we first discuss  
 a suitable SAV algorithm for minimization,  introduce its relaxed version RSAV, and discuss the adaptive rule and choices for the non-negative operator $\mathcal{L}(\theta)$. Then, we present in Section \ref{sec: num}  several numerical results to show the performance of the RSAV in different optimization problems. 
We provide  a convergence study in Section \ref{sec-convergence}, some concluding remarks in
Section \ref{sec-remark}.

\section{A new SAV approach and its relaxed version}
\label{sec-rsav}
We have observed in the last section that the standard SAV approach is not suitable for the minimization problem \eqref{P}. In this section, we propose a different SAV approach and its related version  which are well suited for \eqref{P}.
\subsection{A modified SAV approach}
 We still assume the splitting \eqref{splitting}, and rewrite   \eqref{G} as
\begin{equation}\label{grad_flow}
	\theta_t + \mathcal{L} \theta  + \nabla f(\theta) - \mathcal{L} \theta = 0.
\end{equation}
Since $f(\theta)$ in a minimization problem should always  be bounded from below,  there exists $C>0$ such that $f(\theta)>-C$ for all $\theta$. We introduce  $r(t) =  \sqrt{f(\theta)+C}$, and   expand \eqref{grad_flow} to:
\begin{equation}\label{sav-general}
	\begin{cases}
		\theta_t + \mathcal{L} \theta  + \frac{\nabla f(\theta)}{\sqrt{f(\theta)+C}}r - \mathcal{L}\theta= 0, \\
		r_t = \frac{1}{2\sqrt{f(\theta)+C}}( \nabla f(\theta),  \theta_t).
	\end{cases}
\end{equation}
Then, a simple SAV scheme to approximate the above is
\begin{equation}\label{scheme-general}
	\begin{cases}
		\frac{\theta_{k+1} - \theta_{k}}{\delta t} +  \mathcal{L}  \theta_{k+1} + \frac{\nabla f(\theta_{k})}{\sqrt{f(\theta_{k})+C}}r_{k+1} - \mathcal{L}\theta_{k} = 0, \\
		\frac{r_{k+1} - r_{k}}{\delta t} = \left( \frac{\nabla f(\theta_{k})}{2\sqrt{f(\theta_{k})+C}},  \frac{\theta_{k+1} - \theta_{k}}{\delta t}\right).
	\end{cases}
\end{equation}
Note that if $\theta_{k} \rightarrow  \theta_*$ and $r_{k}\rightarrow r_*$, then \eqref{scheme-general} leads to $\nabla f(\theta_*)=0$, and consequently $\theta_*$ is a solution of \eqref{P}.

The scheme \eqref{scheme-general} leads to a coupled linear system for $(\theta_{k+1}, r_{k+1})$, but it can be implemented explicitly after solving a linear system  as will be shown in the Section \ref{sec-explicit-formula}.
Let $A=I+\delta t \mathcal L$, then \eqref{scheme-general} can be equivalently implemented as
\begin{align*}
{r}_{k+1} & = \frac{1}{1+\delta t  \frac{ ( \nabla f(\theta_k), A^{-1} \nabla f(\theta_k))}{2[f(\theta_k)+C]}}r_k,\\
	\theta_{k+1} & = \theta_k-\frac{{r}_{k+1} }{\sqrt{f(\theta_k)+C}} \delta t A^{-1}\nabla f(\theta_k).
\end{align*}

Moreover,	taking the discrete inner product of the first (resp. second) equation in \eqref{scheme-general} with $\theta_{k+1} - \theta_{k}$ (resp. $2r_{k+1}$), summing up the results, we obtain 
	the following:
	\begin{thm}\label{thm0}
		If $\mathcal{L} $ is non-negative, then for any $\delta t> 0$, the modified energy $r^2$ in the scheme  \eqref{scheme-RSAV}  is unconditionally  diminishing in the sense that
		\begin{equation*}
			r^2_{k+1}-  r^2_{k}= -\frac1{\delta t}\|{\theta_{k+1} - \theta_{k}}\|^2-( \mathcal{L}(\theta_{k+1} - \theta_{k}), (\theta_{k+1} - \theta_{k}))-(r_{k+1}-r_{k})^2.
		\end{equation*}
	\end{thm}

	The above result shows the key advantage of \eqref{scheme-general}: the energy dissipation holds for any $\delta t> 0$ and any splitting with  non-negative  $\mathcal{L} $.

As we shall demonstrate {in numerical tests in Section \ref{sec: num}}, when the cost functional $f(\theta)$ has a suitable splitting, the above algorithm usually converge much faster the vanilla gradient decent method. When  the cost function does not have any obvious quadratic part,  we can either choose any suitable non-negative  linear operator $\mathcal{L}$ in \eqref{splitting}, 
or simply take $ \mathcal{L}=0$, which results in 
a fully explicit method:
\begin{equation}\label{sav-explicit}
	\begin{cases}
		\frac{\theta_{k+1} - \theta_{k}}{\delta t} +  \frac{\nabla f(\theta_{k})}{\sqrt{f(\theta_{k})+C}}r_{k+1} = 0, \\
		\frac{r_{k+1} - r_{k}}{\delta t} = \left( \frac{\nabla f(\theta_{k})}{2\sqrt{f(\theta_{k})+C}},  \frac{\theta_{k+1} - \theta_{k}}{\delta t}\right),
	\end{cases}
\end{equation}
which,  we refer as {\em the SAV gradient descent method}.  As will be shown in the Section \ref{sec-explicit-formula}, the scheme \eqref{sav-explicit}  can be decoupled and implemented as  
\begin{equation}
	\label{explicitiSAV-GD0}
	\begin{split}
		&	r_{k+1}=\frac{r_k}{1+\delta t   \frac{ (\nabla f(\theta_k),\nabla f(\theta_k)) }{2(f(\theta_k)+C)}} ,\\
		&\theta_{k+1}=\theta_k-{\delta t}\frac{r_{k+1}}{ \sqrt{f(\theta_k)+C}} \nabla f(\theta_{k}).
	\end{split}
\end{equation}
Compared with the vanilla gradient descent method \eqref{Vanilla-GD},  there are extra computational costs of { computing both $f(\theta_k)$} and $ (\nabla f(\theta_k),\nabla f(\theta_k)) $ in \eqref{explicitiSAV-GD0}, but  Theorem \ref{thm0} ensures stability for any $\delta t$. In contrast, $\delta t$ in \eqref{Vanilla-GD} needs to be small enough to ensure stability. 

\subsection{A relaxed version of the modified SAV approach}
While for fixed $\delta t$, the solution of the  SAV scheme \eqref{scheme-general} converges to a solution of the minimization problem \eqref{P}, the evolution of $r_{k+1}$ is not directly linked to $\sqrt{f({\theta}_{k+1})+C}$, and its value may decrease rapidly to ensure stability when  $\|\nabla f(\theta_k) - \mathcal{L}\theta_k\|$ becomes large. In this case, the ratio $\frac{r_{k+1}}{\sqrt{f({\theta}_{k+1})+C}}$ may deviate significantly from $1$,  which indicates that $r_{k+1}$ is no longer a good approximation of $\sqrt{f({\theta}(t_{k+1}))+C}$, thus $\theta_{k+1}$ will not be  a good approximation of $\theta(t_{k+1})$. For dynamic simulation of gradient flows, a remedy is to monitor the ratio $\frac{r_{k+1}}{\sqrt{f({\theta}_{k+1})+C}}$ and adjust the time step so that it stays close to 1.  For the minimization problem \eqref{P},  since we are mainly interested in the  steady steady state solutions of \eqref{G},  it is found in \cite{SZ_JCP20} that setting $r_{k+1}=\sqrt{f({\theta}_{k+1})+C}$ at each time step is also  very effective. More precisely, we can use the following modified SAV scheme:
\begin{equation}\label{mSAV}
    \begin{cases}
    \frac{\theta_{k+1} - \theta_{k}}{\delta t} +  \mathcal{L} \theta_{k+1} + \frac{\nabla f(\theta_{k})}{\sqrt{f({\theta}_{k})+C}}\tilde{r}_{k+1} - \mathcal{L}\theta_{k}= 0, \\
    \frac{\tilde{r}_{k+1} - r_{k}}{\delta t} = \left( \frac{\nabla f({\theta}_{k})}{2\sqrt{f({\theta}_{k})+C}},  \frac{\theta_{k+1} - \theta_{k}}{\delta t}\right),\\
    r_{k+1} = \sqrt{f({\theta}_{k+1})+C}.
    \end{cases}
\end{equation}   
However, the above modified SAV scheme is no longer energy diminishing. Recently, another way to link  $r_{k+1}$ with $\sqrt{f(\theta_{k+1})+C}$ while still being   energy diminishing is proposed in \cite{MR4254503} (see also \cite{jiang2022improving,zhang2022generalized}). When applied to \eqref{grad_flow}, the relaxed SAV method takes the following form:

\begin{equation}\label{scheme-RSAV}
    \begin{cases}
    \frac{\theta_{k+1} - \theta_{k}}{\delta t} +  \mathcal{L} \theta_{k+1} + \frac{\nabla f({{\theta}}_{k})}{\sqrt{f({\theta}_{k})+C}}\tilde{r}_{k+1} - \mathcal{L}\theta_{k} = 0, \\
    \frac{\tilde{r}_{k+1} - r_{k}}{\delta t} = \left( \frac{\nabla f({\theta}_{k})}{2\sqrt{f({\theta}_{k})+C}},  \frac{\theta_{k+1} - \theta_{k}}{\delta t}\right),\\
    r_{k+1} = \xi\tilde{r}_{k+1} + (1-\xi)\sqrt{f(\theta_{k+1})+C}.
    \end{cases}
\end{equation}
Here,  the relaxation parameter $\xi$ is a scalar chosen from the set
\begin{equation}\label{set}
    \mathcal{V}= \{\xi\in [0,1] :(r_{k+1})^2 - (\tilde{r}_{k+1})^2 - (\tilde{r}_{k+1} - r_{k})^2\leq  \eta \mathcal{G}(\theta_{k+1}, \theta_{k})\}
\end{equation}
where $\mathcal{G}(\theta_{k+1}, \theta_{k})= \frac{1}{\delta t}\left((\theta_{k+1} - \theta_{k}),A(\theta_{k+1}-\theta_{k})\right)$ with $A = I + \delta t\mathcal{L}$, and $\eta\in[0,1]$ is an artificial parameter with default value $\eta = 0.99$. In particular, it is shown in \cite{jiang2022improving} that we can choose
\begin{equation}
  \xi= \max\{0,\frac{-b-\sqrt{b^2-4ac}}{2a}\},
\end{equation}
with the coefficients that
\begin{equation}
    a = (\tilde{r}_{k+1} - \sqrt{f(\theta_{k+1})+C})^2
\end{equation}
\begin{equation}
    b = 2\left(\tilde{r}_{k+1} - \sqrt{f(\theta_{k+1})+C}\right)\sqrt{f(\theta_{k+1})+C}
\end{equation}
\begin{equation}
    c = f(\theta_{k+1})+C - (\tilde{r}_{k+1} )^2 - (\tilde{r}_{k+1} - r_k)^2 - \eta \mathcal{G}(\theta_{k+1},\theta_{k}).
\end{equation}

Taking the discrete inner product of the first (resp. second) equation in \eqref{scheme-RSAV} with $\theta_{k+1} - \theta_{k}$ (resp. $2\tilde{r}_{k+1}$), summing up the results,
we get
\begin{equation}
\label{eq-g-inner-prod}
	\mathcal{G}(\theta_{k+1},\theta_{k}) = \frac{1}{\delta t}\left((\theta_{k+1} - \theta_{k}),A(\theta_{k+1}-\theta_{k})\right) = -2(\tilde{r}_{k+1} - r_{k})\tilde{r}_{k+1},
\end{equation}
then the non-zero choice of $\xi$ can be rewritten as 
\begin{eqnarray*}
	\xi &=& \frac{-b-\sqrt{b^2-4ac}}{2a} = \frac{\sqrt{f(\theta_{k+1})+C}-\sqrt{ (\tilde{r}_{k+1} )^2 + (\tilde{r}_{k+1} - r_k)^2 + \eta \mathcal{G}(\theta_{k+1},\theta_{k})}}{\sqrt{f(\theta_{k+1})+C} - \tilde{r}_{k+1}}\\
	&=&\frac{\sqrt{f(\theta_{k+1})+C}-\sqrt{(1-\eta)\tilde{r}_{k+1}^2 + \eta r_{k}^2 + (1-\eta)(\tilde{r}_{k+1} -r_{k})^2}}{\sqrt{f(\theta_{k+1})+C}-\tilde{r}_{k+1}}.
\end{eqnarray*}

The implementation of \eqref{scheme-RSAV} is summarized in \textbf{Algorithm 1}.

\begin{thm}\label{thm1}
If $ \mathcal{L} $ is non-negative and  linear, then for any $\delta t> 0$, the modified energy $r^2$ in the scheme  \eqref{scheme-RSAV}  is unconditionally  diminishing in the sense that
\begin{equation*}
   r_{k+1}^2- r_{k}^2= -(1-\eta)\mathcal{G}(\theta_{k+1}, \theta_{k})\leq 0.
\end{equation*}
\end{thm}

\begin{proof}
		By \eqref{eq-g-inner-prod}, we obtain 
		\begin{equation*}
			\tilde r^2_{k+1}-  r^2_{k}= -\frac1{\delta t}\|{\theta_{k+1} - \theta_{k}}\|^2-( \mathcal{L}(\theta_{k+1} - \theta_{k}), (\theta_{k+1} - \theta_{k}))-(\tilde r_{k+1}-r_{k})^2.
		\end{equation*}
Adding $r_{k+1}^2 - \tilde{r}^2_{k+1}$ on both sides, noticing 
	$$\mathcal{G}(\theta_{k+1}, \theta_{k}) = \frac{1}{\delta t}\|\theta_{k+1} - \theta_{k}\|^2 + (\mathcal{L}(\theta_{k+1}-\theta_{k}),(\theta_{k+1}-\theta_{k})),$$ 
we obtain
	\begin{align*}
  r_{k+1}^2- r_{k}^2= - \mathcal{G}(\theta_{k+1}, \theta_{k}) - (\tilde{r}_{k+1} - r_{k})^2 + r_{k+1}^2 - \tilde{r}^2_{k+1},
	\end{align*}
which implies the desired result thanks to (\ref{set}).
\end{proof}

\begin{algorithm}
\caption{The basic RSAV scheme}\label{rsav general implementation}
\begin{algorithmic}[1]
\Inputs{
        $\delta t$: step-size, \\
        $C$: constant to guarantee  the positivity of $f(x)+C$,\\
        $A = I + \delta t  \mathcal{L}$ : the linear operator,\\
        $\theta_0$: initial parameter vector
        }
 \State{
        $r_0\leftarrow \sqrt{f(\theta_0)+C}$
        }
\For{$k=0,1,2,..., M-1$}
\vspace{2pt}
\State $g_{k} = \frac{\nabla f(\theta_{k})}{\sqrt{f(\theta_{k})+C}}$
\State $\hat{g}_{k} = A^{-1}g_{k}$
\State $\tilde{r}_{k+1} = \frac{r_{k}}{1+ \frac{\delta t}{2}(g_k, \hat{g}_k)}$
\State $\theta_{k+1} = \theta_{k} - \delta t \tilde{r}_{k+1} \hat{g}_{k}$
\State {$\xi = \frac{\sqrt{f(\theta_{k+1})+C}-\sqrt{(1-\eta)\tilde{r}_{k+1}^2 + \eta r_{k}^2 + (1-\eta)(\tilde{r}_{k+1} -r_{k})^2}}{\sqrt{f(\theta_{k+1})+C}-\tilde{r}_{k+1}}$}
\State $\xi = \max\{0,\xi\}$
\State $r_{k+1} = \xi\tilde{r}_{k+1} + (1-\xi)\sqrt{f(\theta_{k+1})+C}$
\vspace{2pt}
\EndFor
\Return $\theta_{M}$
\end{algorithmic}
\end{algorithm}


\subsection{Selection of the operator \texorpdfstring{$\mathcal{L}$}{}}
In the SAV approach for gradient flows \cite{shen2019new}, it is found that a proper  choice of the splitting \eqref{splitting}, i.e., the choice of the quadratic term $\frac 12(\mathcal{L}\theta,\theta)$, can significantly increase the robustness and efficiency of the SAV schemes. For gradient flows coming from materials science or fluid dynamics, there are usually obvious candidates in the free energy. However,  for minimization problems, particularly those from machine learning problems, there are usually no obvious quadratic terms in the energy functions.  In these cases, we can artificially choose some simple operators.
 In this paper, we consider two simple operators below, for which the inverse operator $(I+\delta t \mathcal L)^{-1}$ can be efficiently implemented.
\subsubsection{Diagonal Matrix}
In many optimization problems,   an $l^2$ regularization term is often added into the loss function to avoid overfitting to the data in training sets, namely, 
\begin{equation}
	J(x) = f(x) + \frac{\lambda}{2} \|x\|^2.
\end{equation}
In this case, a natural choice is to set  $\mathcal{L} = \lambda I$. More generally, we can use  $\mathcal{L} =D$ with $D$ being a diagonal matrix with positive entries, e.g., $D$ can be the diagonal entries of the Hessian of the cost function. 

\subsubsection{Discrete Laplacian Matrix}
In some machine learning problems, the discrete  Laplacian matrix  is used as a smoothing operator which can reduce the variance during the mini-batch training process  \cite{LSGD}. This corresponds to  $\mathcal{L} = -\sigma \Delta$ where  $\sigma$ is a positive parameter and $\Delta$ is a discrete Laplacian matrix, {and $(I+\delta t \mathcal L)^{-1}$ can be efficiently inverted by FFT based methods.}
The acceleration by using discrete Laplacian in classical primal dual algorithms has been also justified in \cite{jacobs2019solving}.

\subsection{An adaptive algorithm based on the RSAV scheme \texorpdfstring{\eqref{scheme-RSAV}}{}}
Similar to the  modified  SAV scheme \eqref{scheme-general}, the RSAV scheme \eqref{scheme-RSAV} is also unconditionally energy diminishing. A main advantage of unconditionally stable schemes is that one can adaptively adjust the time step size to achieve faster convergence. 
In particular, we can  use 
\begin{equation}
	I_k(r,\theta) = \frac{r_k}{\sqrt{f(\theta_k)+C}}
\end{equation}
 as an indicator   to control the deviation between modified energy and true energy. 
 {For solving differential equations $\theta_t=-\nabla f(\theta
 )$, $	I_k(r,\theta)$ should be as close as to $1$ for the sake of the time accuracy. But for a minimization problem, there is no time accuracy issue thus we can allow $I_k(r,\theta)$ to deviate from $1$ to achieve faster convergence. However, $I_k(r,\theta)$ needs to be away from zero to avoid slow convergence,}
as the SAV and RSAV algorithms may {converge much slower than the vanilla gradient descent} when the ratio $I_k(r,\theta)$ becomes too small.

We observe from (\ref{scheme-RSAV}) that the true  step-size for the gradient $\nabla f(\theta_k)$ is $\frac{\tilde r_{k+1}}{\sqrt{f(\theta_k)+C}}\delta t$.  Thus if the ratio is small i.e. $I_k(r,\theta) < \gamma$,  {the true  step-size for the gradient would be too small resulting in slow convergence.} To this end, we present a simple adaptive rule with an adaptive constant $\rho>1$ with default value $\rho=1.1$ described in {\bf Algorithm 2}.

\begin{algorithm}[!ht]
	\caption{The adaptive RSAV scheme}\label{adaptive rsav general implementation}
	\begin{algorithmic}[1]
		\Inputs{
			$\delta t_0$: initial step-size, $\delta t_{min}$: the lower bound of step-size,\\
			$C$: constant to guarantee the positivity of $f(x)+C$,\\
			$A = I + \delta t \mathcal{L}$ : the linear operator,\\
			$\theta_0$: initial parameter vector,\\
			$\rho$: adaptive constant which is greater than $1$,\\
			$\gamma$: threshold for the indicator $I(r,\theta)$.
		}
         \State{
        $r_0\leftarrow \sqrt{f(\theta_0)+C}$: Initialize the SAV,
        }
		\For{$k=0,1,2,..., M-1$}
		\vspace{2pt}
		\If{$\frac{r_{k}}{\sqrt{f(\theta_k)+C}} < \gamma$ \text{ and } $\delta t > \delta t_{min}$}  
		\State $\delta t_{k+1} = \max\{\frac{r_k}{\sqrt{f(\theta_k)+C}} \delta t_{k}, \delta t_{min}\}$   
		\Else
		\State $\delta t_{k+1} = \rho \delta t_{k}$
		\EndIf
  \\
		\State $g_{k} = \frac{\nabla f(\theta_k)}{\sqrt{f(\theta_{k})+C}}$
        \State $\hat{g}_{k} = A^{-1}g_{k}$
		\State $\tilde{r}_{k+1} = \frac{r_{k}}{1+ \frac{\delta t_{k+1}}{2}(g_{k}, \hat{g}_{k})}$
		\State $\theta_{k+1} = \theta_{k} - \delta t_{k+1} \tilde{r}_{k+1} \hat{g}_{k}$
		\State {$\xi = \frac{\sqrt{f(\theta_{k+1})+C}-\sqrt{(1-\eta)\tilde{r}_{k+1}^2 + \eta r_{k}^2 + (1-\eta)(\tilde{r}_{k+1} -r_{k})^2}}{\sqrt{f(\theta_{k+1})+C}-\tilde{r}_{k+1}}$}
		\State $\xi = \max\{0,\xi\}$
        \State $r_{k+1} = \xi\tilde{r}_{k+1} + (1-\xi)\sqrt{f(\theta_{k+1})+C}$
		\vspace{2pt}
		\EndFor
		\Return $\theta_{M}$
	\end{algorithmic}
\end{algorithm}
\begin{remark}
Note that in many applications of neural networks and machine learning, the cost of computing the full batch is generally  too high. In these cases, we can adopt the mini-batch approach commonly used in stochastic gradient decent, and restart the RSAV scheme at the beginning of each mini-batch. 
\end{remark}

\begin{remark}
As a further generalization, we may   replace the   operator $\mathcal{L}$ in  \eqref{scheme-general} and \eqref{scheme-RSAV} by a linear nonnegative operator $\mathcal{L}_k$ which may depend on $\theta_k$ at each step. Then, Theorems \ref{thm0} and   \ref{thm1} still hold with $\mathcal{L}$ replaced by  $\mathcal{L}_k$. 

\end{remark}

\section{Numerical Results}
\label{sec: num}
We present in this section several illustrative numerical experiments by using  our RSAV approach, and compare it with popular gradient based approaches.

In order to present a fair comparison to gradient descent (GD), we consider {a composite gradient method}.  {By abuse of notation,
we shall refer to {\it GD with $\mathcal L$} as the following method for the splitting \eqref{splitting}:}
\[\theta_{k+1}+\delta t \mathcal{L}\theta_{k+1}=\theta_{k}+\delta t ( \mathcal{L}\theta_{k}-\nabla f(\theta_k)),\]
which is equivalent to
\begin{equation}
	\label{Vanilla-GD2}
	\theta_{k+1}=\theta_{k}-\delta t(I+\delta t \mathcal L)^{-1}\nabla f(\theta_k).
\end{equation}
{
The scheme \eqref{Vanilla-GD2}
can also be regarded as the forward-backward splitting scheme $$\theta_{k+1}=\theta_k-\delta t \nabla F(\theta_k)-\delta t \nabla G(\theta_{k+1})$$ for minimizing a composite function $F(\theta)+G(\theta)$ with $F(\theta)=f(\theta)-\frac12 \theta^T \mathcal L\theta$ and $G(\theta)=\frac12 \theta^T \mathcal L\theta$.}

{Note that the scheme \eqref{Vanilla-GD2} reduces to the vanilla gradient descent \eqref{Vanilla-GD} if setting $\mathcal L=0$, and \eqref{Vanilla-GD2} reduces to the vanilla gradient descent \eqref{Vanilla-GD} with step size $\frac{\delta t}{1+\delta t}$ if setting $\mathcal L=I$. Therefore, we do not consider {GD with $\mathcal L=I$}, and we compare the adaptive RSAV in \textbf{Algorithm 2} to the following algorithms:
\begin{enumerate}
    \item {\it GD with $\mathcal L=0$}, which is the vanilla gradient descent \eqref{Vanilla-GD}.
    \item {\it GD with $\mathcal L=-\sigma \Delta$ with discrete Laplacian $\Delta$}, which is similar to the  Laplacian Smoothing Gradient Descent \cite{LSGD}.
    \item ADAM  \cite{kingma2014adam} 
    \item Nesterov accelerated gradient decent (NAG) \cite{nesterov1983method}.     
\end{enumerate}
Unless specified otherwise, we use ADAM and NAG with the default parameter settings as in \cite{ruder2016overview}:  NAG with $\gamma = 0.9$, and  ADAM with $\beta_1 = 0.9, \beta_2 = 0.999, \varepsilon = 10^{-8}$.}
\subsection{A quadratic cost function}
We  start with  a quadratic loss function from \cite{LSGD}:
\begin{equation}
    f(\theta_1, \theta_2, \ldots, \theta_{100}) = \sum_{i=1}^{50}\theta^2_{2i-1} + \sum_{i=1}^{50}\frac{1}{100}\theta_{2i}^2.
\end{equation}

For this simple function, we take either $\mathcal L=0$ or $\mathcal L=D$ where the diagonal matrix $D$ is chosen to be the Hessian $\nabla^2f(\theta)$.

To demonstrate the unconditional stability of SAV-based approaches, in Table \ref{table-quadratic} and Figure \ref{quad(no noise)} we show the results of different initial step sizes $\delta t$ for the vanilla gradient descent, i.e., GD with $\mathcal L=0$, as well as GD with $\mathcal L=\mathcal D$. 
We observe that the vanilla gradient descent blows up for the constant step size $\delta t = 1$, while the adaptive RSAV works quite well. 

\begin{table}[!ht]
    \begin{tabular}{|c|c|c|c|}
        \hline
           (initial) step-size $\delta t$ &    $0.01$ & $0.1$ & $1$\\ 
        \hline
           GD ($\mathcal{L} = 0$)&   0.3351 & 0.009121  &  50 \\
        \hline
           adaptive RSAV ($\mathcal{L} = 0$)&   6.34e-12   & 5.749e-12  & 2.264e-18  \\
        \hline
           GD ($\mathcal{L} = D$) &    0.3352 & 0.009194 &   3.152e-18  \\
        \hline
        	adaptive RSAV ($\mathcal{L} = D$)&   0   & 0  & 0  \\
        \hline
    \end{tabular}
    \caption{Loss of quadratic function after 1000 iterations.}
    \label{table-quadratic}
\end{table}

\begin{figure}[H]
	\centering
	\begin{subfigure}[b]{0.45\textwidth}            
		\includegraphics[width=\textwidth]{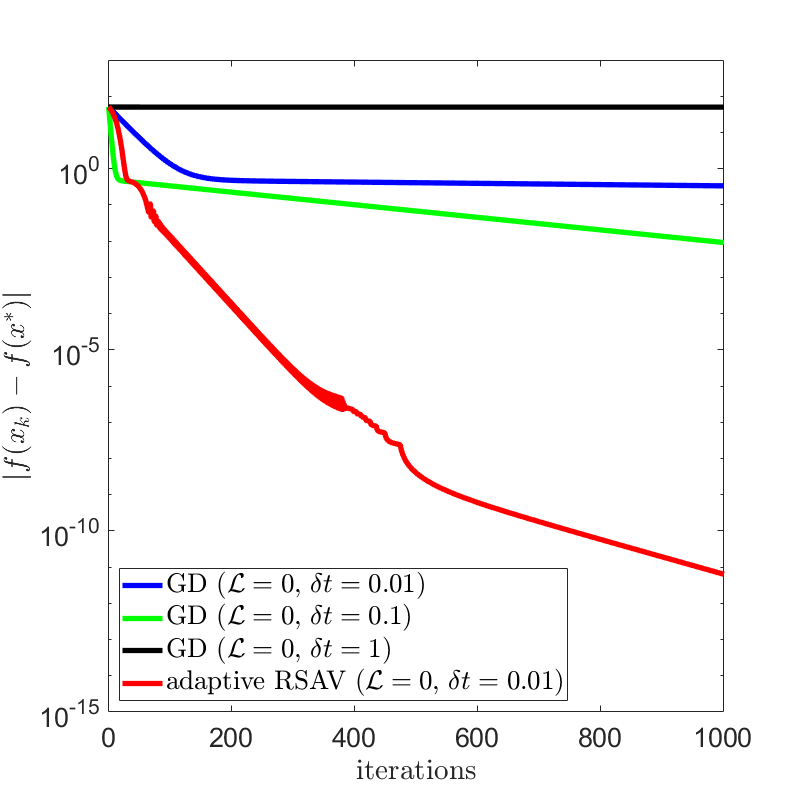}
	\end{subfigure}
	\begin{subfigure}[b]{0.45\textwidth}
		\centering
		\includegraphics[width=\textwidth]{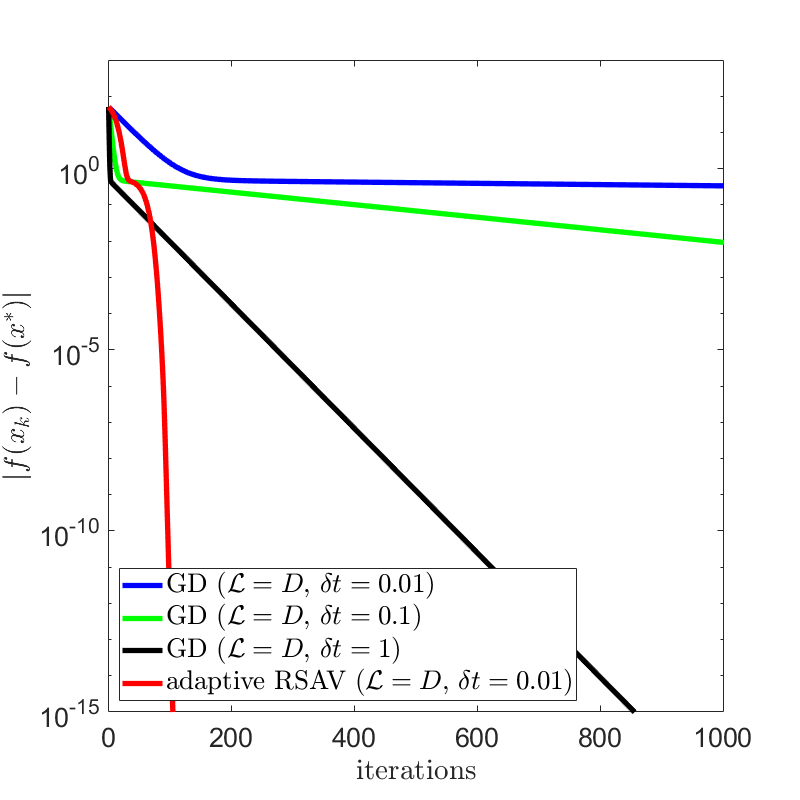}
	\end{subfigure}
	\caption{Loss curves for GD and adaptive RSAV with different splits and (initial) step-sizes $\delta t$.}
	\label{quad(no noise)}
\end{figure}

Next 
we consider the gradient perturbed by a Gaussian noise
\begin{equation}
	\nabla_{\epsilon} f(x) := \nabla f(x) + \epsilon \mathcal{N}(0,I),
\end{equation}
where  $\epsilon$ controls the noise level, $\mathcal{N}(0,I)$ is the Gaussian noise vector with zero mean and unit variance in each coordinate. 
The comparison is given in Table \ref{table-quad-noise} and Figure \ref{quad(noise)} where $\mathcal L=0$ is used for both GD and adaptive RSAV. 
We observe that the adaptive RSAV converges much faster than GD. The fast convergence of adaptive RSAV is partly due to the indicator $I_{k}(r,\theta)$ which can give a proper step size. Especially in the noisy case, the true step size is given {at a proper level} to reach a better convergence than GD and reduce the oscillation in the loss curves.

\begin{table}[!ht]
    \begin{tabular}{|c|c|c|c|}
        \hline
          (initial) step-size $\delta t$ &   $0.01$ & $0.1$ & $1$ \\ 
        \hline
           GD  ($\epsilon = 0.01$)&    0.335 & 0.009223 &   58.58 \\
        \hline
           adaptive RSAV ($\epsilon = 0.01$)&   0.0002283  & 0.0002298 & 0.0002251 \\
        \hline
           GD  ($\epsilon = 0.05$)&    0.3348 & 0.01584 &   diverge \\
        \hline
           adaptive RSAV  ($\epsilon = 0.05$)& 0.004934  & 0.005023  & 0.004889\\
        \hline
           GD  ($\epsilon = 0.1$)&  0.3354 & 0.03808 &   diverge \\
        \hline
           adaptive RSAV  ($\epsilon = 0.1$)& 0.01746   & 0.01924  & 0.0188 \\
        \hline
    \end{tabular}
    \caption{Loss of quadratic function after 1000 iterations with different noise levels $\epsilon$ and (initial) step-sizes $\delta t$. $\mathcal L=0$ is used for both  GD and adaptive RSAV.}
    \label{table-quad-noise}
\end{table}

\begin{figure}[H]
	\centering
	\begin{subfigure}[b]{0.32\textwidth}            
		\includegraphics[width=\textwidth]{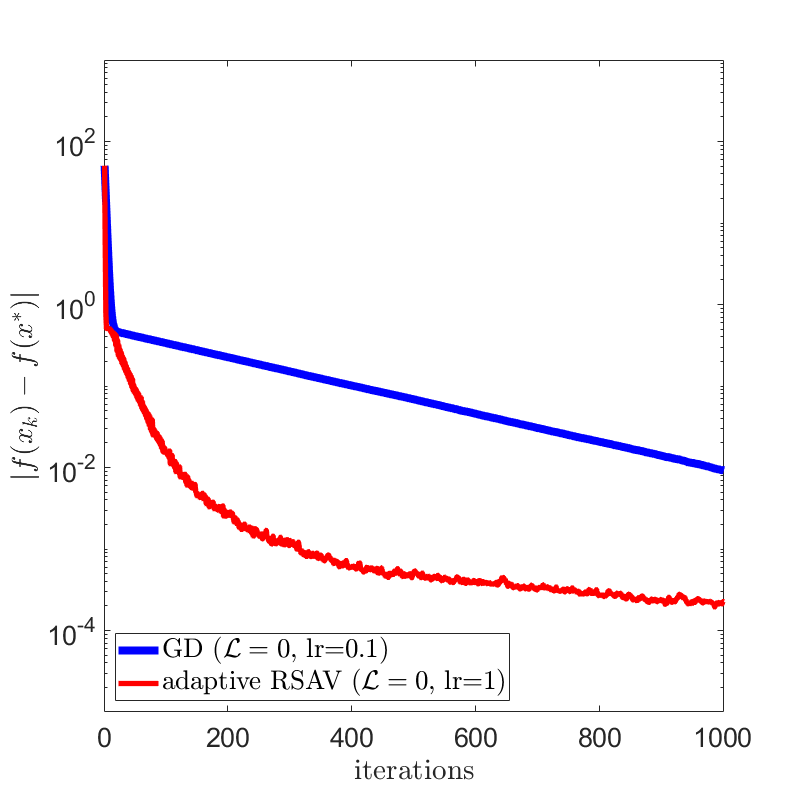}
		\caption{$\epsilon = 0.01$}
	\end{subfigure}
	\begin{subfigure}[b]{0.32\textwidth}
		\centering
		\includegraphics[width=\textwidth]{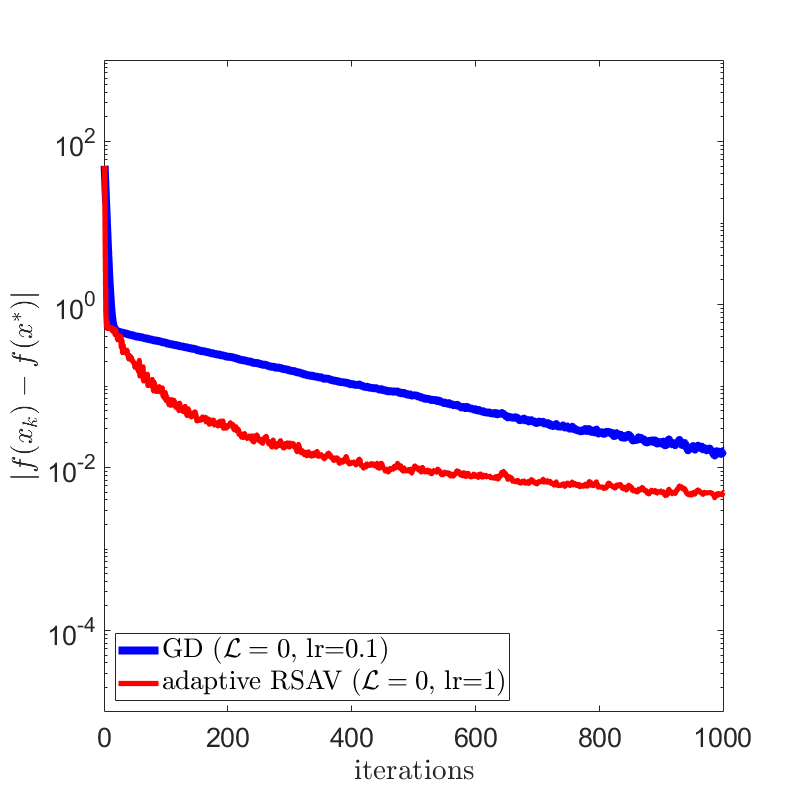}
		\caption{$\epsilon = 0.05$}
	\end{subfigure}
	\begin{subfigure}[b]{0.32\textwidth}            
		\includegraphics[width=\textwidth]{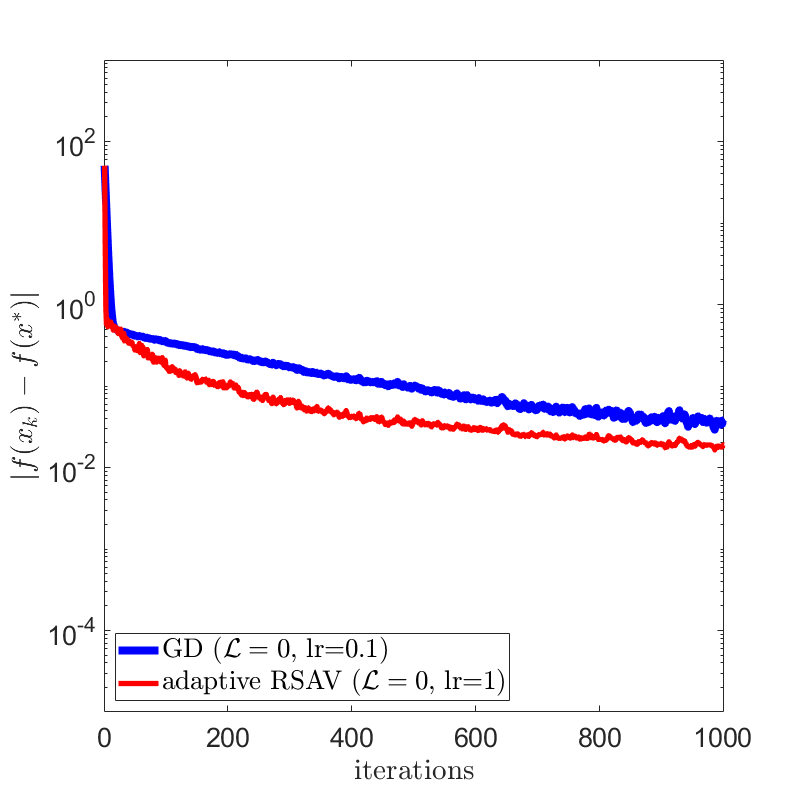}
		\caption{$\epsilon = 0.1$}
	\end{subfigure}
	\caption{Loss curves for GD and adaptive RSAV with different noise levels. The learning rate (lr) refers to the step size $\delta t$ for GD and the initial
 step size $\delta t$ in the adaptive RSAV.}
	\label{quad(noise)}
\end{figure}

\subsection{Rastrigin function}
Consider
\begin{equation}
    f(x) = f(\theta_1, \theta_2, \ldots, \theta_n) = \sum_{i=1}^{n}\theta_i^2 + 10n -10 \sum_{i=1}^{n}\cos(2\pi \theta_i),
\end{equation}
which has many local minima. The function can be defined on any input domain but it is usually evaluated on $x\in [-5.12,5.12]$ for $i=1,2,\ldots,n$. The function has a global minimum at $f(x^{*}) = 0$ located at $x^{*} = (0,0,\ldots,0)$. In this example, we compare the adaptive RSAV with  popular optimization methods ADAM and NAG with their default parameter settings as in \cite{ruder2016overview}:  NAG with $\gamma = 0.9$, and  ADAM with $\beta_1 = 0.9, \beta_2 = 0.999, \varepsilon = 10^{-8}$. We shall keep using these default settings in all following  experiments.

\begin{table}[!ht]
    \begin{tabular}{|c|c|c|c|c|}
        \hline
           initial stepsize $\delta t$ &    $0.001$ & $0.01$ & $0.1$ & $1$\\ 
        \hline
           GD ($\mathcal{L} = 0$)&   12.93 & 37.86  &  109.2  & diverge \\
        \hline
           adaptive RSAV ($\mathcal{L} = 0$)& 13.05  & 13.03  & 12.95   & 2.608e-9\\
        \hline
           NAG &  12.93  &  152 & 505.1  & diverge\\
        \hline
           ADAM &  32.28  &  12.94 & 12.93 & 8.958\\
        \hline
    \end{tabular}
    \caption{Loss of Rastrigin function after 100 iterations in 2D.}
\end{table}
\begin{figure}[!ht]
    \centering           
    $$\includegraphics[width=0.49\textwidth]{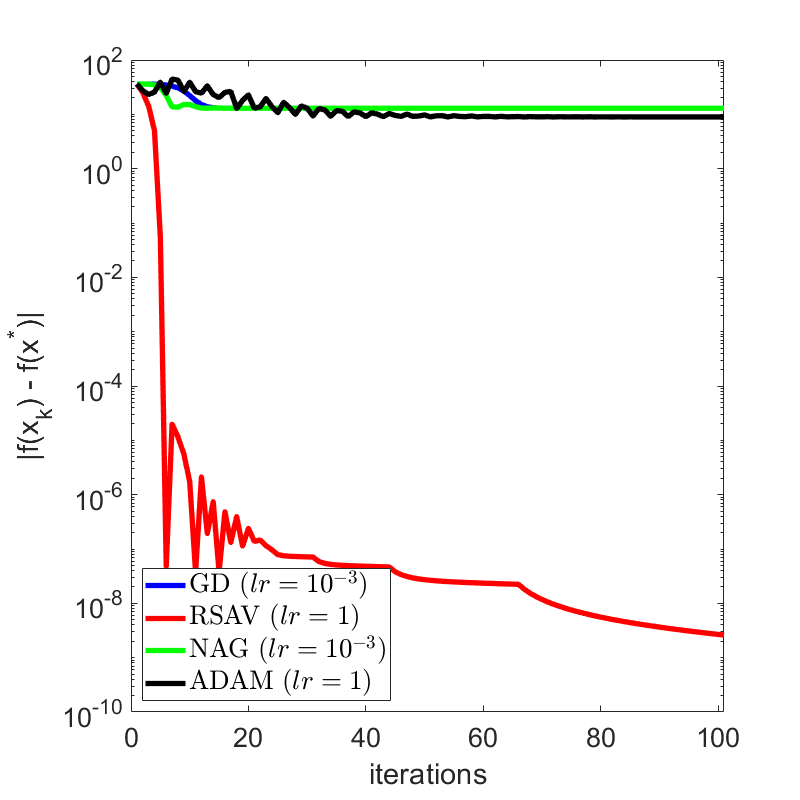}  \includegraphics[width=0.49\textwidth]{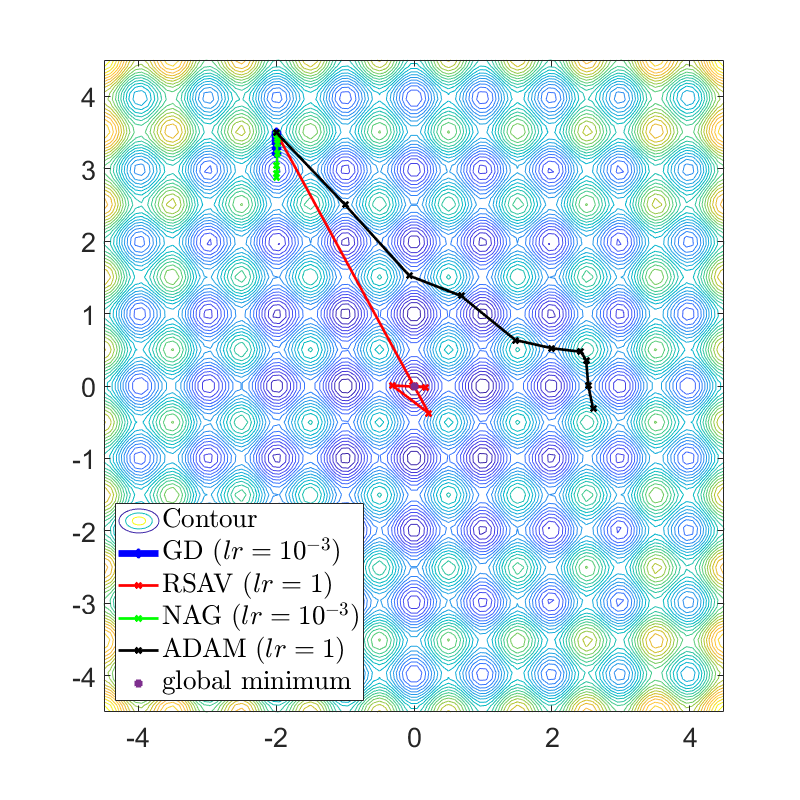}$$
    \caption{Rastrigin function: Left: Convergence curves with 100 iterations; Right: Paths with first 10 iterations. The learning rate (lr) refers to the initial
 step size $\delta t$ in the adaptive RSAV.}\label{Rastrigin}
\end{figure}
We plot in the left of Fig. \ref{Rastrigin} the convergence curves of different methods, and observe that the RSAV converges much faster than ADAM with the same initial step-size. We also plot in the right of Fig. \ref{Rastrigin} the paths towards the minimum by different methods. We observe that RSAV enjoys a fast convergence {if using a large initial step size $\delta t$}.

\subsection{Rosenbrock function}
This is a benchmark problem for optimization of \textbf{non-convex} functions.
We first consider the 2D case with
\begin{equation}
    f(x,y) = (a-x)^2 + b (y-x^2)^2,
\end{equation}
it has a global minimum at $(x,y) = (a,a^2)$, which  is inside a long narrow, parabolic shaped flag valley. To find the valley is trivial, but to converge to the global minimum is usually difficult.

We set $a=1$ and $b=100$ in the following experiments and other parameters the same as in  \cite{LSGD}, and
start with the initial point with coordinate $(-3,-4)$. In Table \ref{table-Rosenbrock}, we observe that a large step size can lead to blow-up for other methods except for RSAV. Thus in Figure \ref{Rosenbrock},  we only show the results with the largest suitable step sizes for other algorithms. For adaptive RSAV, we just use the same initial step size as ADAM. This example reveals the benefits of modified energy decreasing property of the RSAV. Although ADAM can get close to the global minimum at first, it still goes to the wrong direction caused by the momentum and eventually goes back after wasting many iterations. Only RSAV converges to the global minimum directly with the guide of decreasing (modified) energy.

\begin{table}[!ht]
    \begin{tabular}{|c|c|c|c|}
        \hline
           step-size $\delta t$ &    $10^{-4}$ & $10^{-2}$ & $1$\\ 
        \hline
           GD ($\mathcal{L} = 0$)&   0.7142 &   diverge & diverge \\
        \hline
           adaptive RSAV ($\mathcal{L} = 0$)& 0.01086 & 0.01122 & 0.0107\\
        \hline
           NAG &  5.326 & diverge  & diverge\\
        \hline
           ADAM & 15198 & 12.5 & 1.2\\
        \hline
    \end{tabular}
    \caption{Loss of Rosenbrock function after 1000 iterations in 2D}
    \label{table-Rosenbrock}
\end{table}

\begin{figure}[H]
    \centering         
    \includegraphics[width=0.49\textwidth]{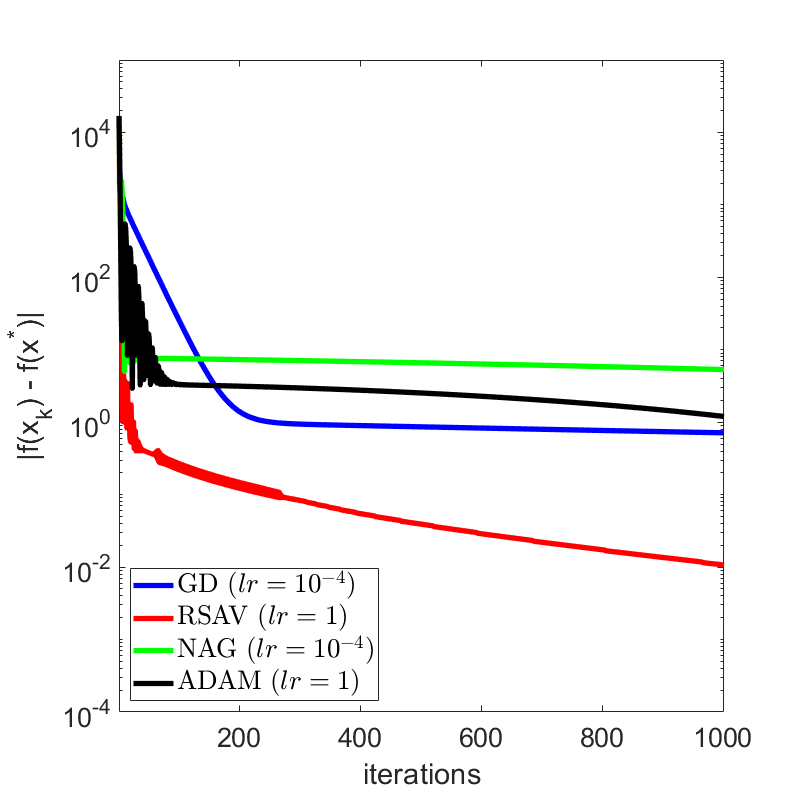}  
    \includegraphics[width=0.49\textwidth]{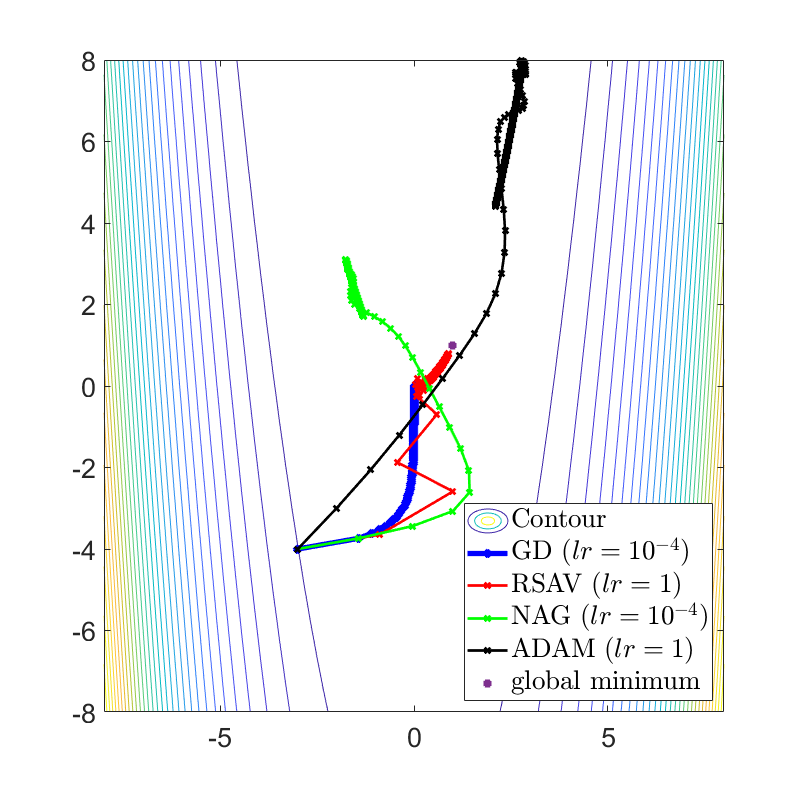}
    \caption{2D Rosenbrock problem: Left: Convergence curves; Right: Paths with 1000 iterations in the  $(\theta_1, \theta_2)$ domain.}\label{Rosenbrock}
\end{figure}

Next, we consider the high dimensional cases with
\begin{equation}
	f(x) = \sum_{i=1}^{n}(a - \theta_i)^2 + b\sum_{i=1}^{n-1}(\theta_{i+1} - \theta_i^2)^2,
\end{equation}
with the global minimum $f(x^{*}) = 0$ at $x^{*}=(a,a^2,a,a^2,\ldots,a,a^2)$. We take $a = 1$ and $b = 100$, and the initial point  $(0,\ldots,0)$.
The results with the dimension equal to 10, 100 and 1000 are shown in Fig. \ref{Rosenbrock(high)}. We observe  similar convergence behavior for all cases as in the two dimensional case.

\begin{figure}[H]
\centering
    \begin{subfigure}[b]{0.32\textwidth}            
            \includegraphics[width=\textwidth]{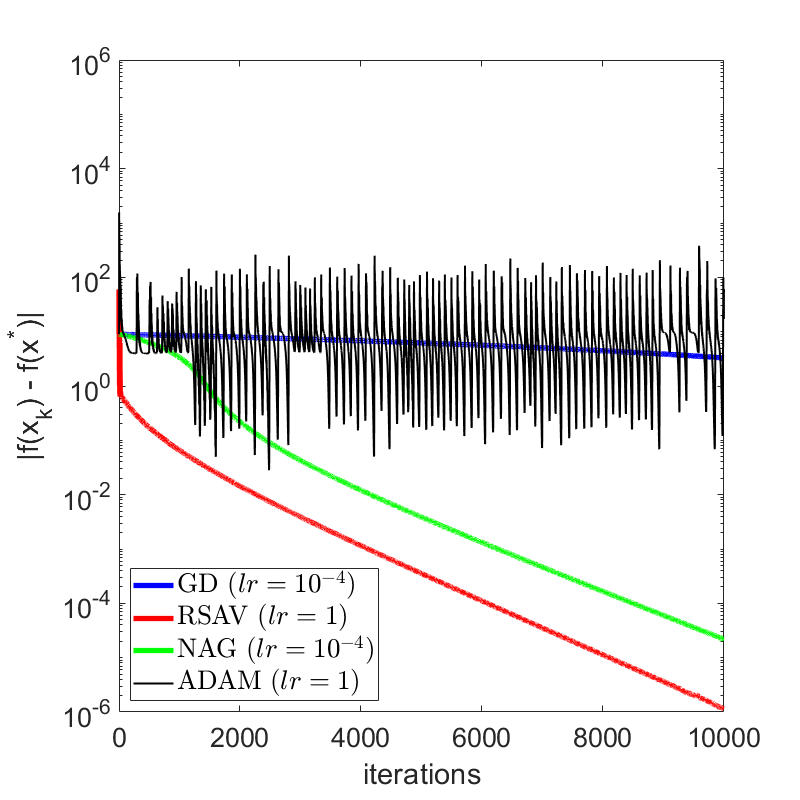}
            \caption{$n = 10$}
    \end{subfigure}
    \begin{subfigure}[b]{0.32\textwidth}
            \centering
            \includegraphics[width=\textwidth]{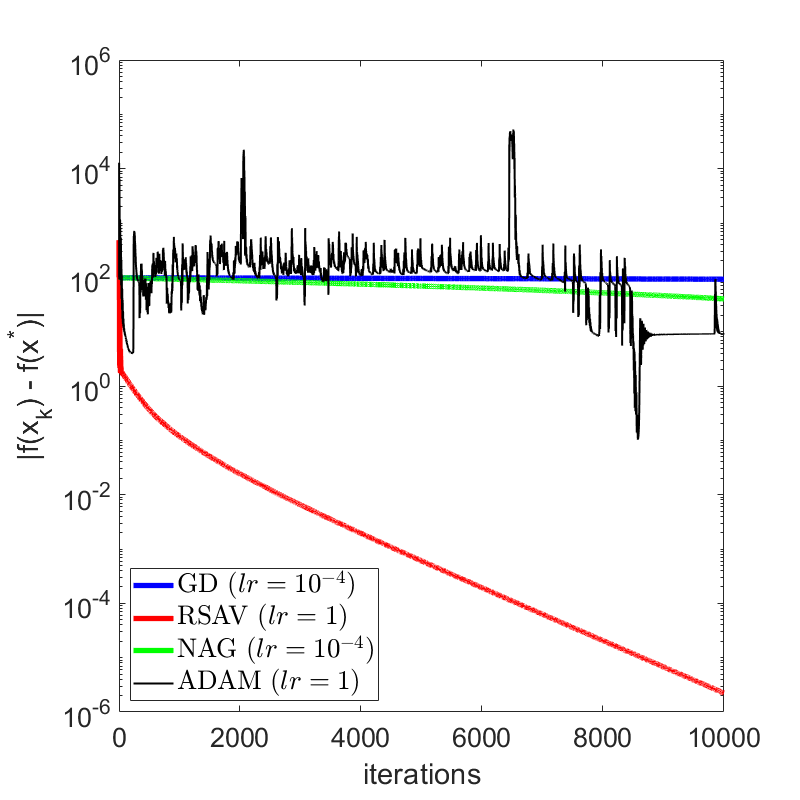}
            \caption{$n= 100$}
    \end{subfigure}
        \begin{subfigure}[b]{0.32\textwidth}            
            \includegraphics[width=\textwidth]{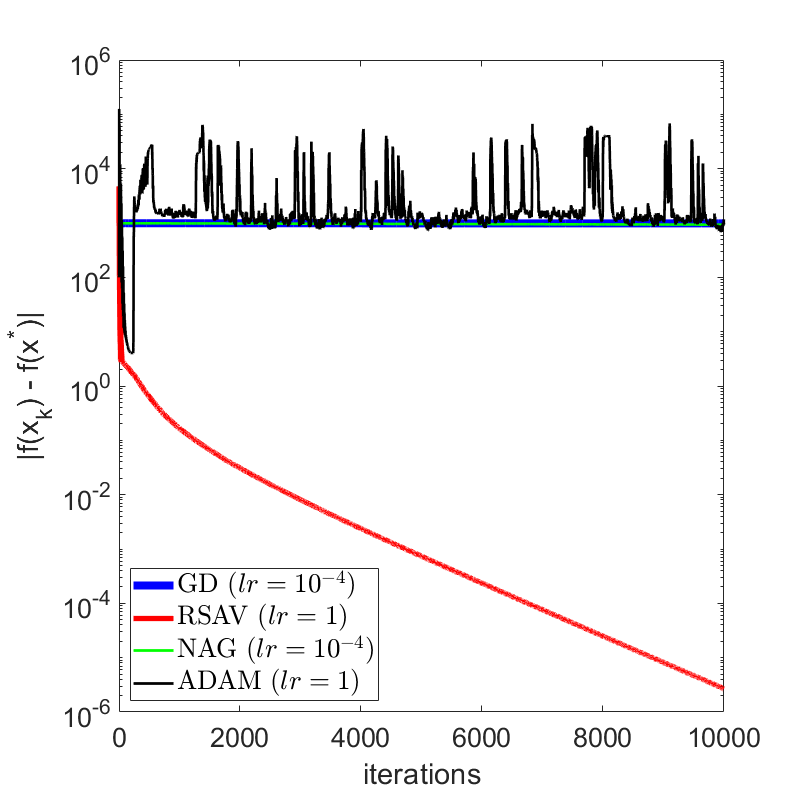}
            \caption{$n = 1000$}
    \end{subfigure}
    \caption{Loss of Rosenbrock function  with dimension $n$.}
    \label{Rosenbrock(high)}
\end{figure}

\begin{remark}
	 If we compare the adaptive RSAV with Nesterov accelerated gradient decent, the results are still quite good especially when the dimension is $1000$. Thanks to the adaptive scheme, the performance of RSAV in this problem independent of the dimension.
\end{remark}
\subsection{Phase Retrieval}

The phase retrieval problem \cite{candes2015phase2} can be formulated as 
$$\min_{z\in \mathbbm C^N} f(z):=\frac12\|\mathcal A(z^* z)-b \|^2$$ where $z^*\in \mathbbm C^{1\times N}$ is the conjugate transpose of $z$,  $b\in\mathbbm R^M$ and $\mathcal A: \mathbbm C^{N\times N}\longrightarrow \mathbbm R^M$ is a linear operator.  
For the real-valued function $f(z)$ with complex variable $z:=a+\mathbbm i b$, where $\mathbbm i$ is the imaginary unit and $a,b\in \mathbbm R$ are real and imaginary parts of $z$, we can define the Fr\'echet derivative induced by the natural choice of real inner product for $\mathbbm C^N$  as the following \cite{HermitianPSD}:
\[\nabla f(z):=\frac{\partial f(z)}{a}+\mathbbm i \frac{\partial f(z)}{b}=2\mathcal A^*(\mathcal A(z^*z-b))z,\]
where $\mathcal A^*$ is the adjoint operator of $\mathcal A$.
Then the vanilla gradient descent 
algorithm for $\min_{z\in \mathbbm C^N} f(z)$  can be defined  as in \eqref{Vanilla-GD} using $\nabla f(z)$ above. 
The gradient descent method with a suitable step sizing rule is also  also referred to as the Wirtinger flow \cite{candes2015phase}.

  In particular, $f(z)$ is a non-convex quartic polynomial function of $z$.  
  For the theorectical convergence of minimizing such a non-convex function,
with a spectral initialization, i.e., $z_0$ being the leading eigenvector of $\mathcal A^*(b)$, the convergence of Wirtinger flow with high probability can be proven  for a very special class of phase retrieval problems \cite{candes2015phase, cai2016optimal}. For solving phase retrieval with random initial guess, the convergence for minimizing a smoothed amplitude flow based model was proven in \cite{cai2022solving}. 
In terms of practical performance with only random initialization, state-of-the-art algorithms such as the Riemannian LBFGS method could be much more efficient than gradient descent algorithms \cite{candes2015phase}.

We emphasize that we only use such phase retrieval problems as a testing example to validate the performance of the RSAV method. So we test the algorithms with a random initialization. 

  We compare vanilla gradient descent (GD), adaptive RSAV with $\mathcal L=0$, and steepest descent (SD) \cite{curry1944method,bryson1962steepest}.
    The steepest descent method is to use the optimial step size in \eqref{Vanilla-GD}, and it is possible to compute such an optimal step size for a polynomial cost function. 
We test the performance of RSAV algorithm on the following phase retrieval problem.
Let $\mathcal M_i\in \mathbbm C^{N}$ be i.i.d Gaussian and $\circ$ denote the entrywise product. 
Let $\mathcal F$ denote the Fourier transform.
The linear operator $\mathcal A$ is defined by assigning
$\|\mathcal F(\mathcal M_i\circ z)\|^2$ to $b$, e.g., $\frac{M}{N}=m.$ 
We consider the test case for the true solution $z_*$ being an image of size $n\times n$ with $n=256$. So the size of unknown is $N=n^2=256^2$.
We consider two test cases: 
\begin{enumerate}
\item The true minimizer $z_*$ is a real image of camera man with size $256\times 256$ as shown in Figure \ref{real-image},  $m=6$ Gaussian random masks and a random initial guess.
\item The true minimizer $z_*$ is a complex image of golden ball  with size $256\times 256$, see \cite{huang2017solving} for details,  $m=10$ Gaussian random masks  and a random initial condition.
\end{enumerate}

{See Figure \ref{phaseretrieval} for the comparision of the performance of gradient based algorithms. 
For the vanilla gradient descent (GD), we use the nearly largest stable constant step size $\delta t=0.5$. 
In Figure \ref{phaseretrieval_step}, we list a comparsion of the step size $\delta_k$ in the adaptive RSAV
method with the optimital step size in the steepest descent method.  
We can see the performance of adaptive RSAV has the closest performance to the steepest descent. For the real image case, SD converged after 10000 iterations and RSAV converged after 20000 iterations. However,  to compute the optimal step size in the steepest descent, at least two more evaluations of $\mathcal A$ are needed, thus quite expensive. More importantly, for a general cost function, it is difficult to find the optimal step size. See Section \ref{sec-quadratic} for an analysis of the step size in the explicit SAV gradient descent with restarting $r_k$ every iteration for quadratic functions.  }

\begin{figure}[H]
\centering
    \begin{subfigure}[b]{0.325\textwidth}            
            \includegraphics[width=\textwidth]{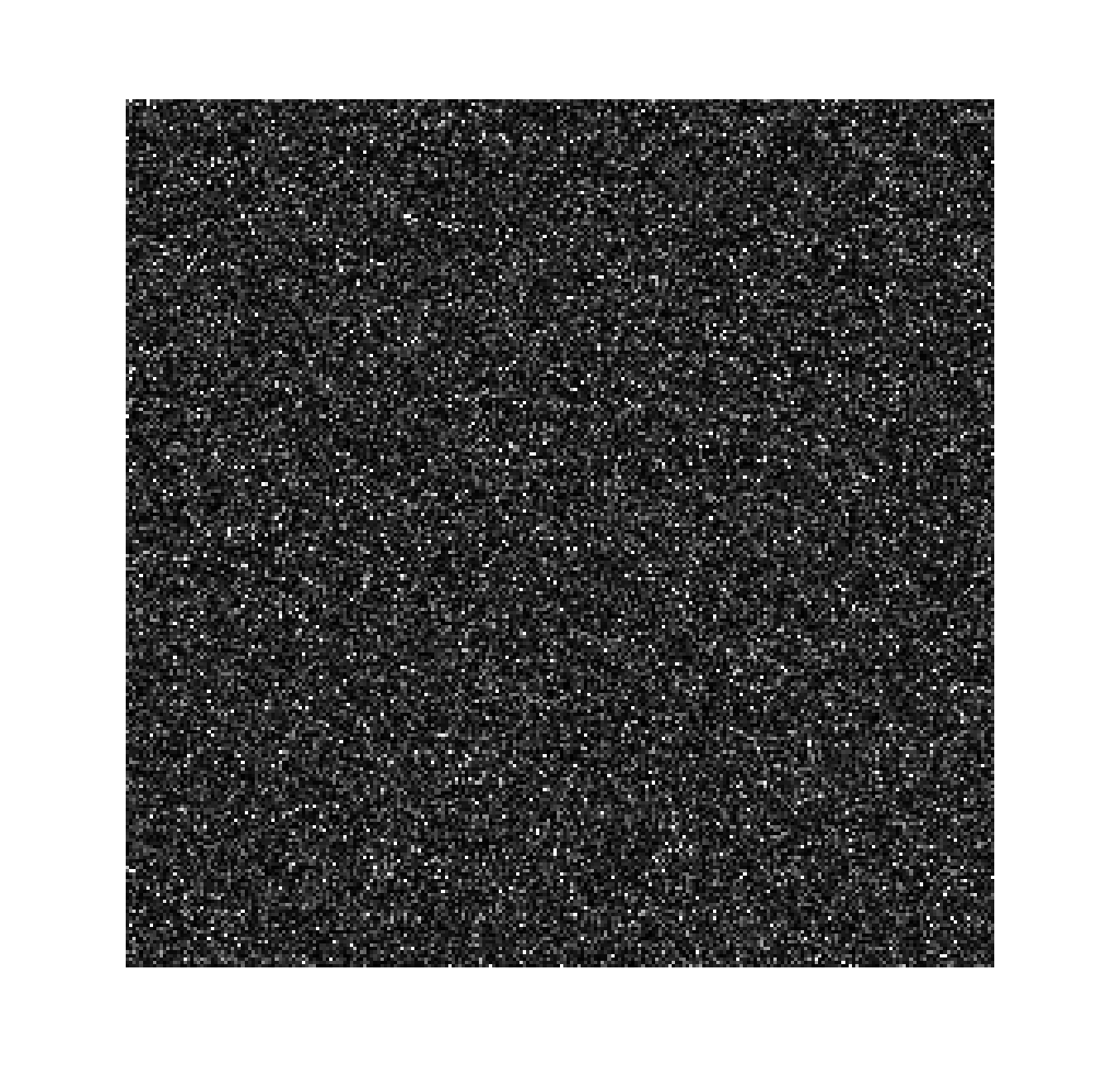}
            \caption{GD}
    \end{subfigure}
    \begin{subfigure}[b]{0.325\textwidth}
            \centering
            \includegraphics[width=\textwidth]{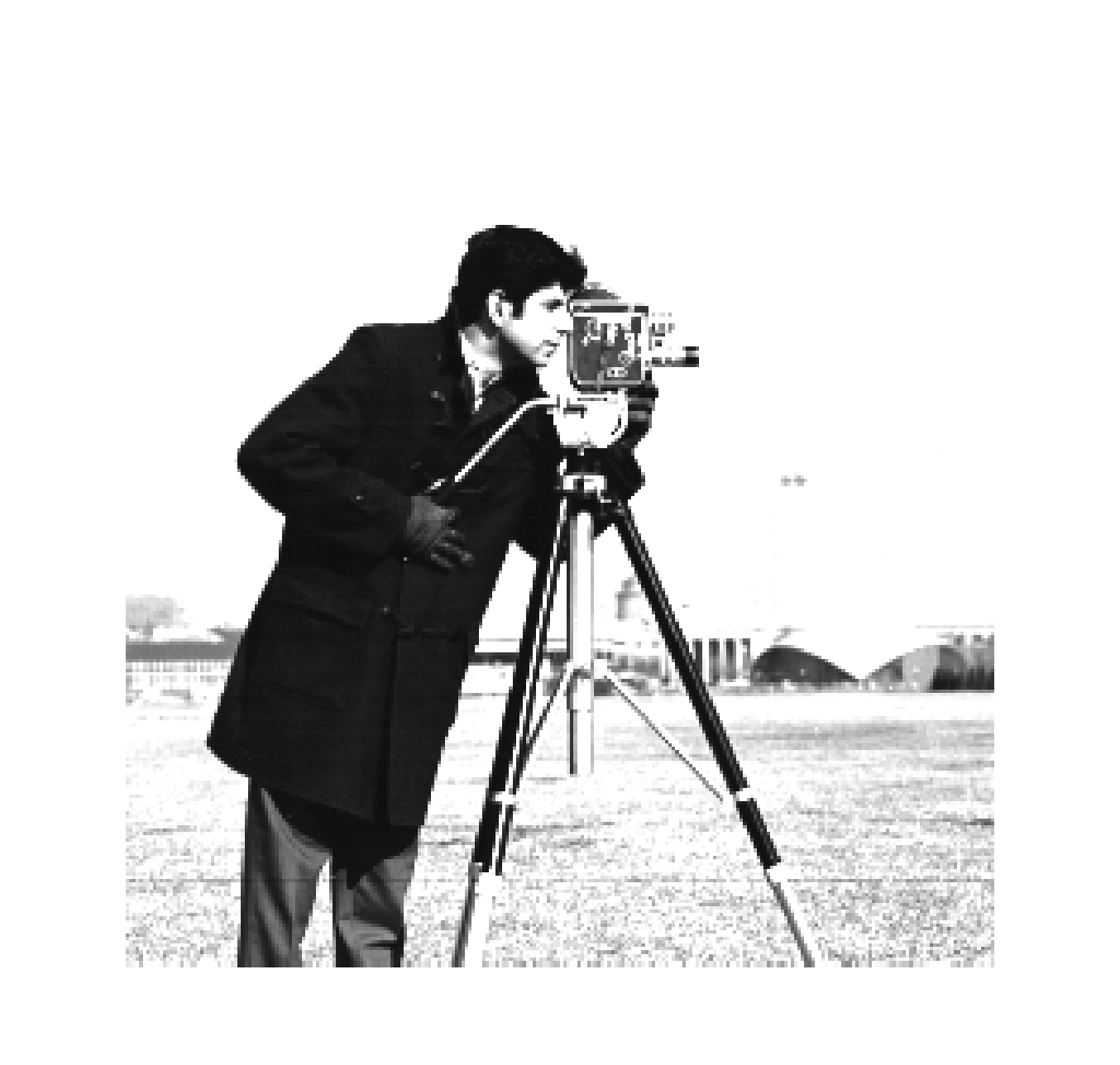}
            \caption{adaptive RSAV}
    \end{subfigure}
        \begin{subfigure}[b]{0.325\textwidth}            
            \includegraphics[width=\textwidth]{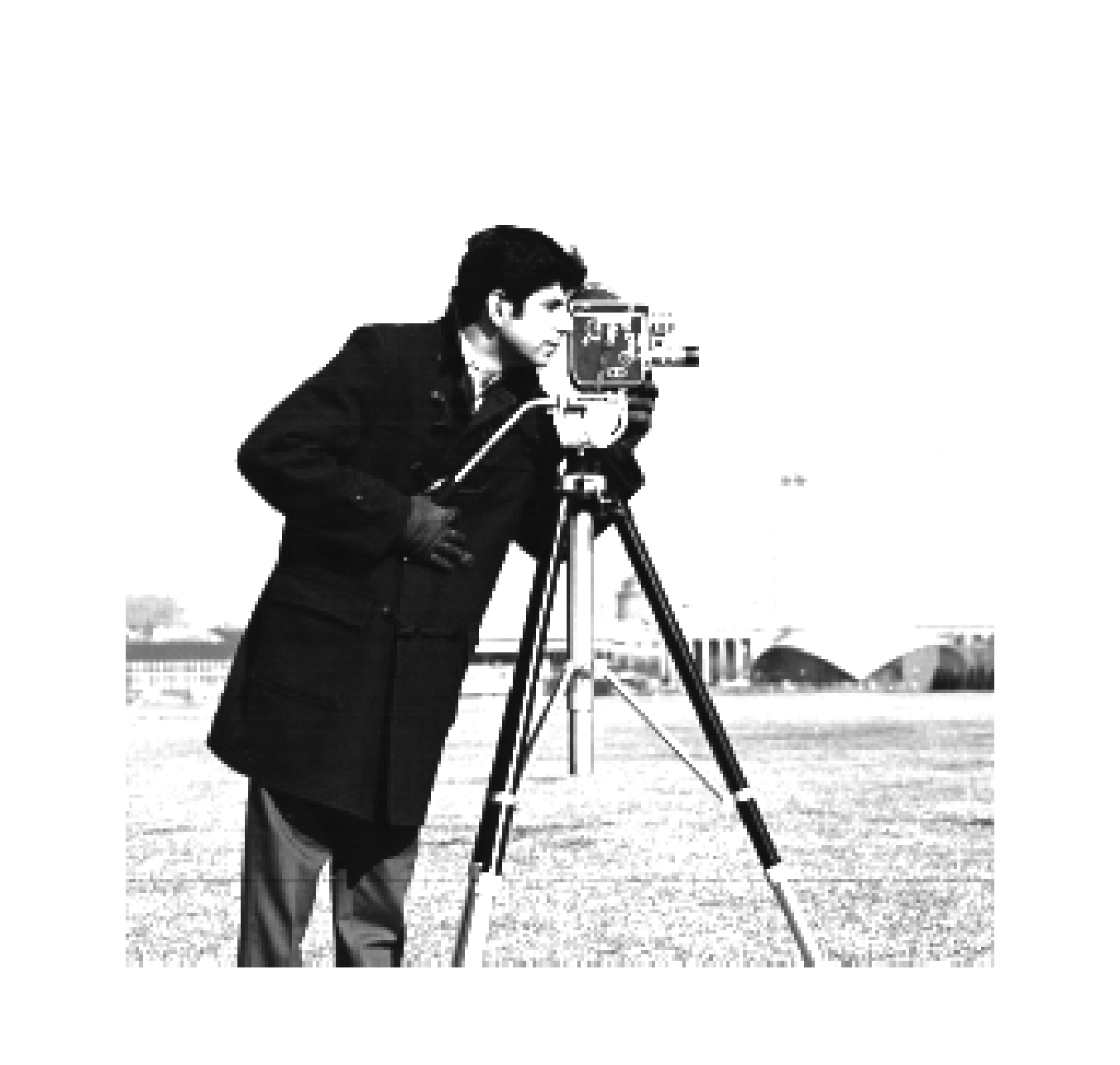}
            \caption{SD}
    \end{subfigure}
    \caption{Results after 20000 iterations for a phase retrieval problem with 
    $z_*$ being a $256\times 256$ real image of camera man,
      $m=6$ Gaussian random masks and a random initial guess. The vanilla gradient descent (GD) uses nearly largest stable constant step size $\delta t=0.5$}
    \label{real-image}
\end{figure}

\begin{figure}[H]            
	\includegraphics[width=0.30\textwidth]{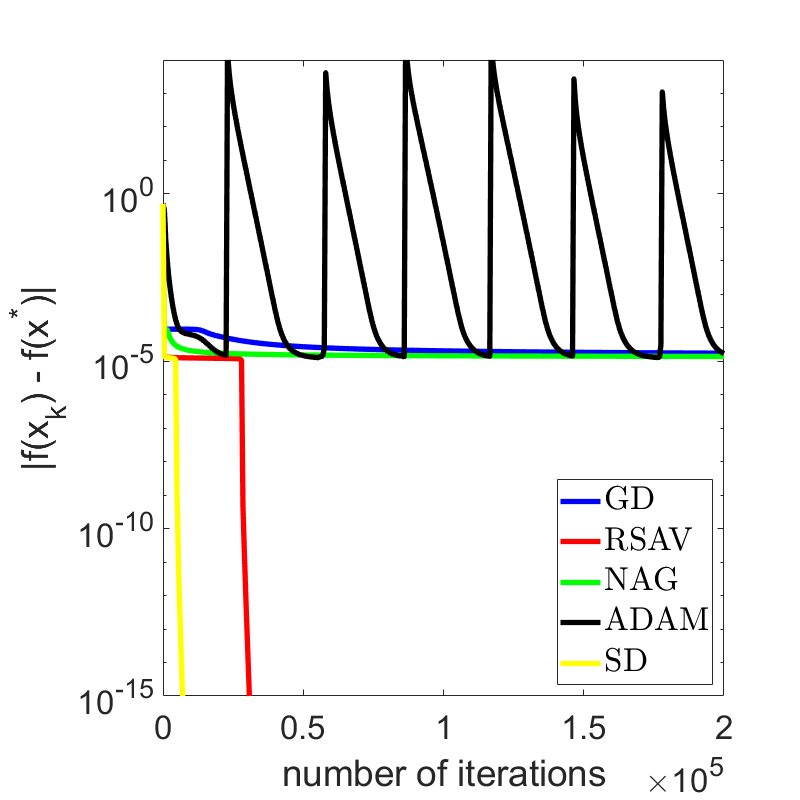}
	\includegraphics[width=0.30\textwidth]{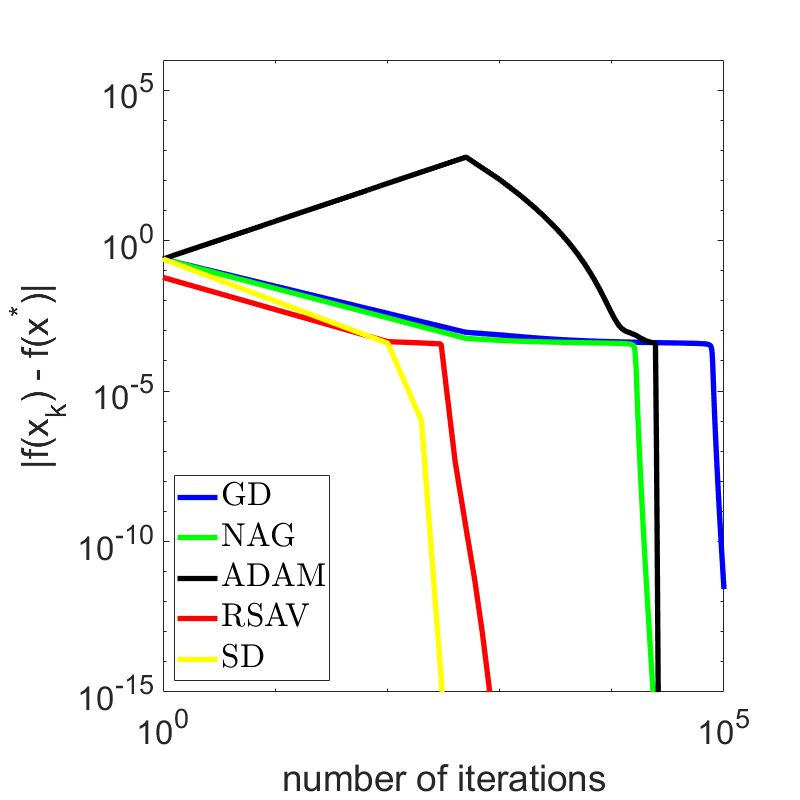}
	\includegraphics[width=0.30\textwidth]{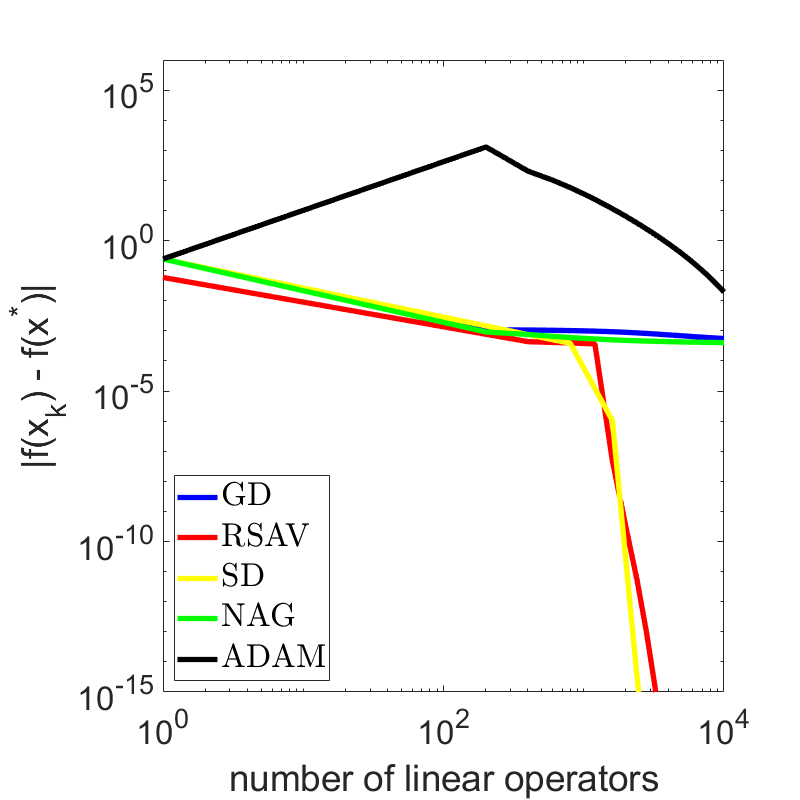}
	\caption{ (Left) Loss of different optimization algorithms V.S. number of iteration  for phase retrieval: the real image of camera man using $6$ Gaussian masks; (Middle) Loss of different optimization algorithms V.S. number of iteration  for phase retrieval: the complex image of golden balls using $10$ Gaussian masks. (Right) Loss of different optimization algorithms V.S. number of iteration  for phase retrieval:  the complex image of golden balls using $10$ Gaussian masks. The vanilla gradient descent (GD) uses nearly largest stable step size $\delta t=0.5$. }\label{phaseretrieval}
\end{figure}

\begin{figure}[H]            
	\includegraphics[width=0.45\textwidth]{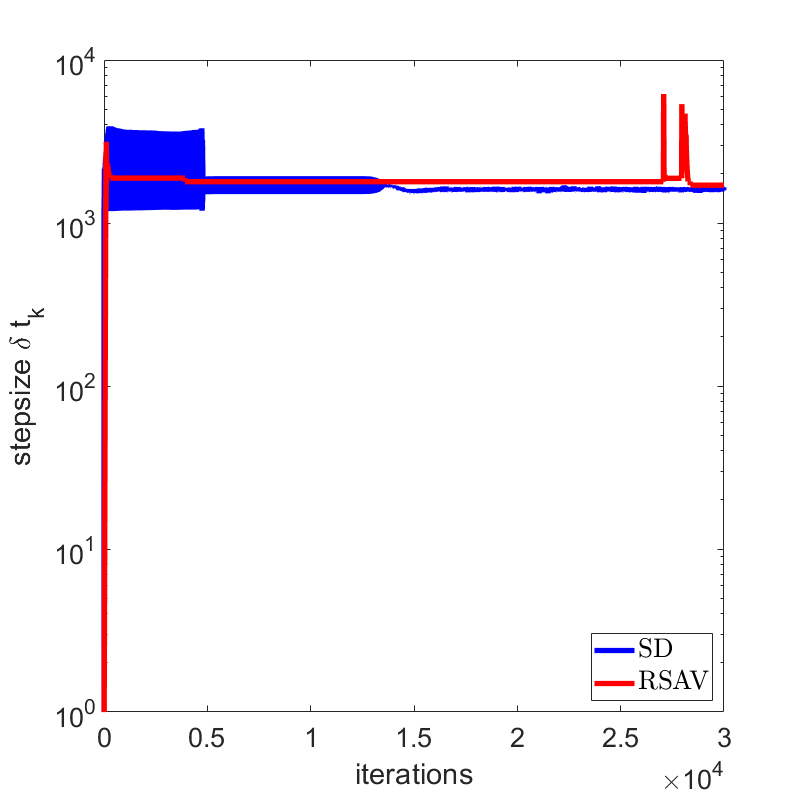}
	\includegraphics[width=0.45\textwidth]{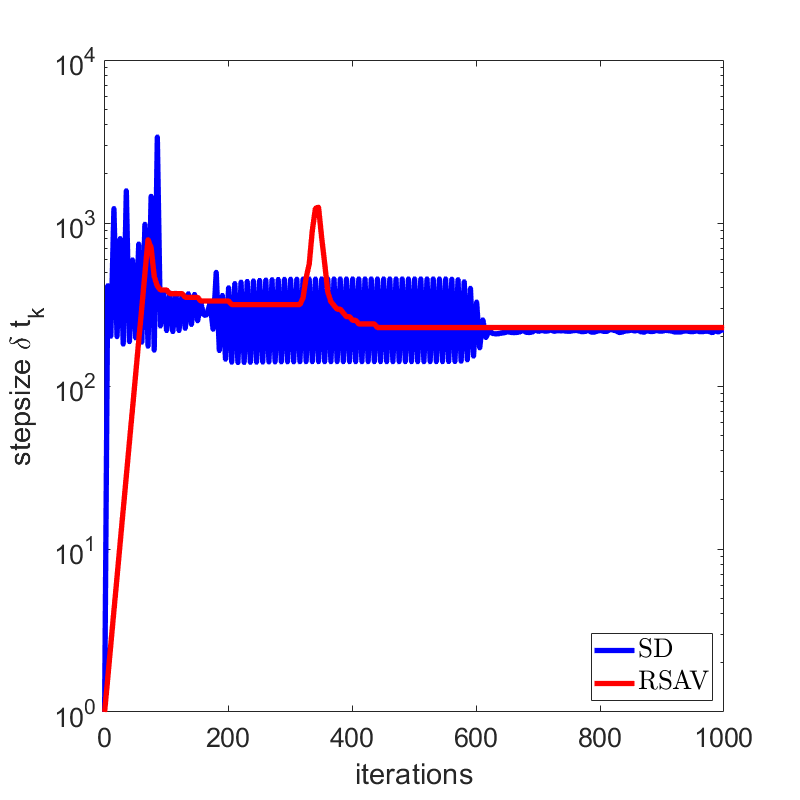}
	\caption{The value of $\delta_k$ in each iteration  for phase retrieval: Left: the real image of camera man using $6$ Gaussian masks; Right: the complex image of golden balls using $10$ Gaussian masks.}\label{phaseretrieval_step}
\end{figure}

\subsection{Recommendation System}
Consider applying the optimization scheme to train a recommendation system based on matrix factorization model. Given a rate matrix $R\in\mathbb{R}^{m\times n}$ where $m$ is the number of users and $n$ is the number of items, the model learns the user embedding matrix $X\in \mathbb{R}^{m\times d}$ and item embedding matrix $Y\in \mathbb{R}^{n\times d}$ such that the product $ XY^T$ is a good approximation for $R$. Here, $d$ is the embedding dimension and usually much smaller than $m$ and $n$. Denote the user and item matrix by $X=[X_1,\ldots,X_u,\ldots, X_m]^T$ and $Y=[Y_1,\ldots,Y_i,\ldots, Y_n]^T$, we have the loss function as 
\begin{equation}
    f(X,Y) = \frac{1}{N_{\kappa}}\sum_{(u,i)\in\kappa}(R_{u,i} - X_uY^T_i)^2 + \lambda_{u}\sum_{u} \|X_u\|^2_2 + \lambda_{i}\sum_{i} \|Y_i\|^2_2,
\end{equation}
where $\kappa$ the training set that the $(u,i)$ pairs for which $R_{u,i}$ is known, $N_{\kappa}$ is the number of training data, $\lambda_u$ and $\lambda_i$ are the penalty parameters for embedding matrix. We train the model with the MovieLens 100K dataset \cite{harper2015movielens} which contains $100,000$ ratings $(1-5)$ from $943$ users on $1682$ movies. There is $80\%$ data split as the training data and the rest date is used for the testing data, e.g., the training data set $\kappa$ has size $80,000$.
All algorithms use the mini-batch gradient with batch size $80$.
For $l_2$ regularization, we set $\lambda_{u} = \lambda_{i} = 10^{-4}$. For the linear operator $\mathcal{L} = \lambda I - \sigma \Delta$ in GD and RSAV, we let $\lambda = 10^{-4}$ and $\sigma = 0.1$.
In  Figure \ref{training}, for the training step, we run $10,000$ iterations for the mini-batch gradient based methods with batch size $80$,
which is equal to $10$ epochs. The result on the test data is shown in Table \ref{table-test-data}.
{We can see that RSAV performs well in the training step, though its testing result is not the best, which suggests issues of overfitting during the training step. This is more or less a modelling issue, rather than the optimizaiton issue.}
\begin{figure}[H]            
    \includegraphics[width=0.6\textwidth]{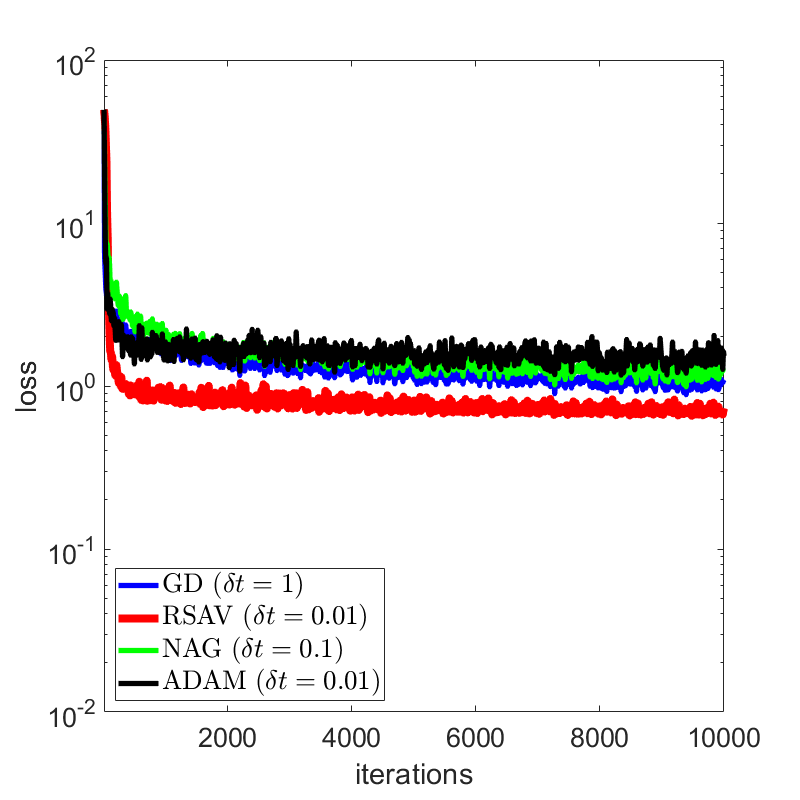}
    \caption{The training  loss curve of different optimization algorithms for Recommendation System.
    Here GD refers to GD ($\mathcal{L} = \lambda I - \sigma \Delta$) and RSAV refers to the adaptive RSAV ($\mathcal{L} = \lambda I - \sigma \Delta$).}
    \label{training}
\end{figure}

\begin{table}[!ht]
    \begin{tabular}{|c|c|c|c|c|}
        \hline
           step-size $\delta t$ &  $0.01$ & $0.1$ & $1$ & $10$\\ 
        \hline
           GD ($\mathcal{L} = \lambda I - \sigma \Delta $)&   4.6504 & 2.1465 &   1.6438 &  diverge\\
        \hline
           NAG &    2.1451 & 1.6439 &   diverge & diverge\\
        \hline
           ADAM &   1.8820 & 4.7194  & diverge & diverge\\
        \hline
           adaptive RSAV ($\mathcal{L} = \lambda I - \sigma \Delta $)&   1.9090 & 1.9102 & 1.9156 & 1.9156\\
        \hline
    \end{tabular}
    \caption{The loss function on the test data after $10,000$ training iterations (10 epochs) with different step-sizes. Here {\it diverge} means that the training step already diverges.}
    \label{table-test-data}
\end{table}

\section{Convergence study for some SAV based algorithms}
\label{sec-convergence}
In this section, we consider a more general version of the SAV scheme based on  the following expanded system
\begin{equation}\label{savq-general}
	\begin{cases}
		\theta_t=-\left(\frac{r}{[f(\theta)+C]^q}\nabla f(\theta)+\mathcal{L}\theta- \mathcal{L} \theta\right)\\
		r_t=q [f(\theta)+C]^{q-1} ( \nabla f(\theta), \theta_t),
	\end{cases}
\end{equation}
where   $r(t)=[f(\theta)+C]^q$ and $q\in (0,1)$. Note that \eqref{sav-general} is a special case of the above formulation with $q=\frac12$.
Similar to  \eqref{scheme-general}, we can construct a SAV scheme for \eqref{savq-general} as follows:
\begin{equation}\label{RSAVq}
	\begin{cases}
		\frac{\theta_{k+1}-\theta_k}{\delta t}=-\left(\frac{{r}_{k+1}}{[f(\theta_k)+C]^{q}}\nabla f(\theta_k) + \mathcal{L}(\theta_{k+1} - \theta_{k})\right)\\
		\frac{{r}_{k+1}-r_{k}}{\delta t}=q[f(\theta_k)+C]^{q-1} ( \nabla f(\theta_k), \frac{\theta_{k+1}-\theta_{k}}{\delta t}).
	\end{cases}
\end{equation}
\subsection{Interpretation of the SAV method as a line search method}
\label{sec-explicit-formula}
Let $A=(I+\delta t \mathcal{L})$. 
The system \eqref{RSAVq} can be rewrite as 
\begin{equation*} 
	\begin{bmatrix}
		A & \delta t \frac{\nabla f(\theta_k)}{[f(\theta_k)+C]^q}\\
		-q [f(\theta_k)+C]^{q-1}   \nabla f(\theta_k) & 1 
	\end{bmatrix}
	\begin{bmatrix}
		\theta_{k+1}\\
		 {r}_{k+1}
	\end{bmatrix} 
	= 
	\begin{bmatrix}
		A \theta_{k}\\
		r_{k}-q [f(\theta_k)+C]^{q-1} ( \nabla f(\theta_k),  \theta_k ) 
    \end{bmatrix}
\end{equation*}
After a simple Gaussian elimination, we obtain an explicit update formula for 
\eqref{RSAVq}:
\[\begin{cases}
 {r}_{k+1} & = \frac{1}{1+\delta t q \frac{ ( \nabla f(\theta_k), A^{-1} \nabla f(\theta_k))}{f(\theta_k)+C}}r_k\\
	\theta_{k+1} & = \theta_k-\frac{{r}_{k+1} }{[f(\theta_k)+C]^q} \delta t A^{-1}\nabla f(\theta_k)
\end{cases}.\] Notice that the scheme above can be regarded as a line search method:
\begin{subequations}
	\begin{align*}
		\theta_{k+1} &=\theta_k+\alpha_k P_k\\
		P_k &=-A^{-1}\nabla f(\theta_k)\\ 
		\alpha_k &=\frac{\delta t}{1+\delta t q \frac{ ( \nabla f(\theta_k), A^{-1} \nabla f(\theta_k))}{f(\theta_k)+C}}\frac{r_k }{[f(\theta_k)+C]^q} > 0,
	\end{align*}
\end{subequations}
with a search direction $P_k$ and step size $\alpha_k$.

The step size $\alpha_k$ is guranteed to be positive. On the other hand, it is difficult to establish any {\it a priori} control of $\alpha_k$, and in practice $\alpha_k$ could become very  small if $r_k$ becomes very small. 
To avoid small $r_k$, we consider a special version of SAV method by 
redefining $r_k=[f(\theta_k)+C]^q$, then we have
$$\alpha_k = \frac{\delta t}{1+\delta t q \frac{ ( \nabla f(\theta_k), A^{-1} \nabla f(\theta_k))}{f(\theta_k)+C}}.$$
In this case, we can view $q$ as a
parameter, and 
the SAV method with   $r_k=[f(\theta_k)+C]^{q}$ at every iteration becomes the following line search method:
\begin{subequations}
	\label{explicitiSAV-linesearch}
	\begin{align}
		\theta_{k+1}&=\theta_k+\alpha_k P_k\label{explicitiSAV-linesearch1}\\
		P_k&=-A^{-1}\nabla f(\theta_k)\\ 
		\alpha_k&=\frac{\delta t}{1+\delta t q \frac{ ( \nabla f(\theta_k), A^{-1} \nabla f(\theta_k))}{f(\theta_k)+C}},
	\end{align}
\end{subequations}
which is equivalent to
\begin{equation}\label{savq-general-2}
	\begin{cases}
		r_k=[f(\theta_k)+C]^{q}\\
		\frac{\theta_{k+1}-\theta_k}{\delta t}=-\left(\frac{\tilde{r}_{k+1}}{[f(\theta_k)+C]^{q}}\nabla f(\theta_k) + \mathcal{L}(\theta_{k+1} - \theta_{k})\right)\\
		\frac{\tilde{r}_{k+1}-r_{k}}{\delta t}=q[f(\theta_k)+C]^{q-1} ( \nabla f(\theta_k), \frac{\theta_{k+1}-\theta_{k}}{\delta t}).
	\end{cases}
\end{equation}


In particular, for any $\mathcal{L}\geq 0$,  $A^{-1}$ is always positive definite, thus the search direction $P_k=-A^{-1}\nabla f(\theta_k)$  is always a descent direction, i.e., $-\nabla f^T(\theta_k) P_k=\nabla f(\theta_k)^TA^{-1}\nabla f(\theta_k)> 0$. 
The Wolfe condition \cite{wolfe1969convergence} for the line search method \eqref{explicitiSAV-linesearch} is: there exists  $0<c_1<c_2<1$ such that 
\begin{subequations}\label{Wolfe}
	\begin{align}
		f(\theta_k+\alpha_k P_k)&\leq f(\theta_k)+c_1\alpha_k \nabla f(\theta_k)^T P_k\\
		\nabla f(\theta_k+\alpha_k P_k)^T P_k&\geq c_2 \nabla f(\theta_k)^T P_k.
	\end{align}
\end{subequations}
We recall first  the following result  \cite{nocedal2006numerical}:
\begin{thm}
	Assume $f(\theta)\in C^1$ and $f(\theta)$ is bounded from below. For any descent direction $P_k$,
	there exist intervals of step lengths satisfying the Wolfe condition.
\end{thm}

Notice that $\alpha(\delta t, q)=\frac{\delta t}{1+\delta t q \frac{ ( \nabla f(\theta_k), A^{-1} \nabla f(\theta_k))}{f(\theta_k)+C}}$ is an increasing function of $\delta t$ and an decreasing  function of $q$, thus 
there exists $\delta t$ and $q$ such that
$\alpha_k=\frac{\delta t}{1+\delta t q \frac{ ( \nabla f(\theta_k), A^{-1} \nabla f(\theta_k))}{f(\theta_k)+C}}$ satisfies the Wolfe conditions \eqref{Wolfe}. 

We recall below another result  \cite{nocedal2006numerical}:
\begin{thm}
	Assume $f(\theta)\in C^1$, $f(\theta)$ is bounded from below and $\nabla f(\theta)$ is Lipschitz continuous. Let $\cos\phi_k=\frac{-\nabla f(\theta_k)^T P_k}{\|\nabla f(\theta_k)\| \|P_k\|}$. If $P_k$ is a descent direction and $\alpha_k$ satisfies the Wolfe Conditions, 
	then the iteration $\theta_{k+1}=\theta_k+\alpha_k P_k$ satisfies 
	\begin{equation}
		\sum_{k\geq 0} \cos^2\phi_k \|\nabla f(\theta_k)\|^2<\infty.
		\label{Zoutendijk}
	\end{equation}
\end{thm}

Let $\lambda_{\min} (A)$ and   $\lambda_{\max} (A)$ be the smallest and largest eigenvalues of the real symmetric positive definite matrix $A$, then by the Courant-Fischer-Weyl min-max principle \cite{courant2008methods} and the spectral norm $\|A\|=\lambda_{\max} (A)$, we have 
\[\frac{P_k^T A P_k}{\|P_k\|^2}\geq \lambda_{\min} (A),\quad \|A P_k\|\leq \|A\| \|P_k\|=\lambda_{\max} (A) \|P_k\|,\]
thus 
\[\cos\phi_k=\frac{P_k^T A P_k}{\|A P_k\| \|P_k\|}=\frac{P_k^T A P_k}{\|P_k\|^2} \frac{\| P_k\|}{\|A P_k\|}\geq \frac{\lambda_{\min} (A)}{\lambda_{\max} (A)}.\]
Therefore, the uniform lower bound on $\cos \phi_k$ and  \eqref{Zoutendijk} implies that $\|\nabla f(\theta_k)\|\rightarrow 0.$

Thus the convergence of the  SAV method  \eqref{explicitiSAV-linesearch} is ensured if using a line search to find $\delta t, q$ such that $\alpha_k$ satisfies the Wolfe condition \eqref{Wolfe}. We observe that the above algorithm involves computing $A^{-1}\nabla f(\theta_k)$, and evaluation of $f(\theta_k)$ and $f(\theta_{k+1})$.  

\begin{remark} 
	In practice, one can use backtracking line search on $\alpha_k$ 
	to ensure that the Wolfe conditions are satisfied. 
	This in general it does not seem advantageous over  a simple  backtracking line search on $\alpha$. However, for the SAV gradient descent method \eqref{sav-explicit}, i.e., $A=I$, our numerical observation is that the SAV scheme is often  more efficient than the backtracking line search on $\alpha$. 
	With $A=I$, the scheme \eqref{explicitiSAV-linesearch} reduces to the following  SAV gradient descent method with two parameters $\delta t>0$ and $q_k>0$:
	\begin{subequations}
		\label{explicitiSAV-GD}
		\begin{align}
			\theta_{k+1}&=\theta_k-\alpha_k \nabla f(\theta_k)\\
			\alpha_k&=\frac{\delta t}{1+\delta t q_k \frac{ \| \nabla f(\theta_k)\|^2}{f(\theta_k)+C}}
		\end{align}
	\end{subequations}
\end{remark}

\subsection{Standard convergence results}
We recall that if $\alpha_k$ in the line search method \eqref{explicitiSAV-linesearch} satisfies the Goldstein-Armijo rule \cite{armijo1966minimization,goldstein1965steepest}:  there exists $0<c_1<c_2<1$ such that

\begin{equation}
	f(\theta_k)-c_2\alpha_k \|\nabla f(\theta_k)\|^2
	\leq f(\theta_k-\alpha_k \nabla f(\theta_k))\leq f(\theta_k)-c_1\alpha_k \|\nabla f(\theta_k)\|^2,
\end{equation}

then it is shown (cf. Theorem 2.1.14  in \cite{nesterov2013introductory}) that   $\|\nabla f(\theta_k)\|\rightarrow 0$.

Theorem 2.1.14 in \cite{nesterov2013introductory} can be easily adapted to prove the following result for the line search method \eqref{explicitiSAV-linesearch}: 
\begin{thm}
	Assume  that $f(\theta)$ is convex and $\nabla f(\theta)$ is Lipshitz continuous with the Lipshitz constant $L$, i.e., $\|\nabla f(y)-\nabla f(x)\|\leq L\|x-y\|$.  
	If $\nabla f(\theta_*)=0$ and $\alpha_k\in(0,\frac{2}{L})$, then 
	\[\|\theta_{k+1}-\theta_*\|^2\leq \|\theta_{k}-\theta_*\|^2-\alpha_k(\frac{2}{L}-\alpha_k)\|\nabla f(\theta_k)\|^2 \]
	and
	\[f(\theta_k)-f(\theta_*)\leq \frac{1}{[f(\theta_0)-f(\theta_*)]^{-1}+\|\theta_0-\theta_*\|^{-2} \sum_{k}\alpha_k(1-\frac{L}{2}\alpha_k) }.\]
\end{thm}
So for convergence, we need $\sum\limits_{k=0}^\infty\alpha_k(1-\frac{L}{2}\alpha_k)=+\infty$, which can be ensured if $\alpha_k\in [a, b]\in (0,\frac{2}{L})$ for constant bounds $a>0$ and $b<\frac{2}{L}$.
Also, $\alpha_k<\frac{2}{L}$ will ensure  $f(\theta_{k+1})<f(\theta_k)$.

\subsection{Decreasing step sizes for the SAV gradient descent method}

We derive from \eqref{explicitiSAV-GD} that

$$\alpha_k=\frac{\delta t}{1+\delta t q_k \frac{ \| \nabla f(\theta_k)\|^2}{f(\theta_k)}}=\frac{1}{1/\delta t +q_k \frac{ \| \nabla f(\theta_k)\|^2}{f(\theta_k)}} \leq \min \left\{ \delta t, \frac{f(\theta_k)}{q_k  \| \nabla f(\theta_k)\|^2}\right\}.$$

For fixed $\delta t$, the above is often sufficient to ensure $f(\theta_{k+1})<f(\theta_k)$ when $\theta_k$ is far away from the minimizer. It can be understood as follows.

\begin{thm}
	
	Assume that  $f(\theta)$ is strongly convex, i.e., 
	$( \nabla f(y)-\nabla f(x), y-x)\geq m\|x-y\|^2$  with $m>0$, and $\nabla f(\theta)$ is Lipshitz continuous with the Lipshitz constant $L$.
	Let $\theta_*$ be the minimizer and assume  $f(\theta_*)=0$. 
	Then the following are sufficient conditions to ensure $\alpha_k<\frac{2}{L}$  for  the SAV gradient descent method  with two parameters \eqref{explicitiSAV-GD}:
	\begin{enumerate}
		\item For any $\delta t>0$,  $q_k> \frac{L^2}{4m^2}$. 
		\item Let $\delta t\equiv a\frac{2}{L}$ where $a>0$, $q_k> \frac{a-1}{a} \frac{L^2}{4m^2}$. 
	\end{enumerate}
	
\end{thm}

\begin{remark}
	The first sufficient condition implies that $\alpha_k=\frac{\delta t}{1+\delta t q_k \frac{ \| \nabla f(\theta_k)\|^2}{f(\theta_k)}}$ will be a decreasing step size for any $\delta t$ if $q_k\equiv q > \frac{L^2}{4m^2}$. 
	Of course, finding $\frac{1}{q} < \frac{4m^2}{L^2}$ in general is not easier than finding $\delta t<\frac{2}{L}$.
	But if $2m^2>L$, then $ \frac{4m^2}{L^2}>\frac{2}{L}$ implies $\frac{1}{q} < \frac{4m^2}{L^2}$ is easier to achieve.  
\end{remark}

\begin{remark}
	As an example of the second sufficient condition, if we pick $q_k\equiv \frac12$, and $a=2$, then $\frac12>\frac{1}{2}\frac{L^2}{4m^2}\Leftrightarrow L<2m$ is sufficient to ensure 
	the SAV gradient descent method with $q_k\equiv \frac12$ is decreasing with $\delta t =\frac{4}{L}$, instead of $\delta t<\frac{2}{L}$ in \eqref{Vanilla-GD}. 
\end{remark}

\begin{proof}
	
	First, by the strong convexity and Lipschitz continuity, we have
	 \[f(x)\geq f(y)+( \nabla f(y), x-y) +\frac{m}{2}\|x-y\|^2,\]
\[f(x)\leq f(y)+( \nabla f(y), x-y) +\frac{L}{2}\|x-y\|^2.\]
	 
	Since $\theta_*$ is the minimizer, $\nabla f(\theta_*)=0$. For any $\theta$, we have
	\[ ( \nabla f(\theta)-\nabla f(\theta_*), \theta-\theta_*)\geq m\|\theta-\theta_*\|^2\Rightarrow ( \nabla f(\theta), \theta-\theta_*)\geq m\|\theta-\theta_*\|^2\]
	\[\Rightarrow m \frac{\|\theta-\theta_*\|}{\|\nabla f(\theta)\|}=\frac{( \nabla f(\theta), \theta-\theta_*)}{\|\theta-\theta_*\|\|\nabla f(\theta)\|}\leq 1\Rightarrow \|\nabla f(\theta)\|\geq m\|\theta-\theta_*\|.\]
	Hence,
	\[   m\|\theta-\theta_*\|\leq \|\nabla f(\theta)\|\leq L\|\theta-\theta_*\|,\]
	and
	\[   \frac {m}{2}\|\theta-\theta_*\|^2\leq f(\theta)-f(\theta_*)\leq \frac {L}{2}\|\theta-\theta_*\|^2.\]
	With strong convexity $m>0$, we have
	\[   \frac{2m^2}{L}\leq \frac{\|\nabla f(\theta)\|^2}{f(\theta)-f(\theta_*)}\leq \frac{2L^2}{m}\]
	thus 
	\[   \frac{2m^2}{L}\leq \frac{\|\nabla f(\theta)\|^2}{f(\theta)}\leq \frac{2L^2}{m}.\]
	Finally,
	
	\[ \frac{f(\theta_k)}{q_k  \| \nabla f(\theta_k)\|^2}<\frac{2}{L}\Leftrightarrow \frac{1}{q_k}< \frac{\|\nabla f(\theta_k)\|^2}{f(\theta_k)}\frac{2}{L}\Leftarrow \frac{1}{q_k}< \frac{4m^2}{L^2}.\]
\end{proof}

\begin{remark}
	In general, if $f(\theta)$ is only convex but not strong convex, i.e., $m=0$, and $f(\theta)+ C>0$, then we only have
	\[\frac{\|\nabla f(\theta)\|^2}{f(\theta)+C} \leq \frac{L^2\|\theta-\theta_*\|^2}{  f(\theta_*)+C}.\]
	This gives a lower bound control on step size:

	\[\alpha_k=\frac{\delta t}{1+\delta t q_k \frac{ \| \nabla f(\theta_k)\|^2}{f(\theta_k)+C}}\geq \frac{\delta t}{1+q_k\delta t\frac{L^2\|\theta_k-\theta_*\|^2}{  f(\theta_*)+C}}.\]

	In this case,  we can set $q_k\equiv q$ and do back tracking on $\delta t$ for $\alpha_k$ to satisfy the convergence condition or Goldstein-Armijo rule.
\end{remark}

\subsection{The step size for quadratic functions}
\label{sec-quadratic}
To see why the step size $\alpha_k=\frac{\delta t_k}{1+\delta t_k q_k \frac{ \| \nabla f(\theta_k)\|^2}{f(\theta_k)+C}}$ could be a good step size to use, at least for a quadratic cost function, consider a 
cost function $f(\theta)=\frac12\|A\theta-b\|^2$ with a square and positive definite matrix $A>0$. The steepest descent algorithm can be written as
\[\theta_{k+1}=\theta_k-\beta_k\nabla f(\theta_k),\quad \beta_k=\frac{\|A^T(A\theta_k-b)\|^2}{[A^T(A\theta_k-b)]^T A^TA [A^T(A\theta_k-b)]}.\]
The method \eqref{explicitiSAV-GD} with $C=0$ is 
\[\theta_{k+1}=\theta_k-\alpha_k\nabla f(\theta_k),\quad \alpha_k=\frac{1}{\frac{1}{\delta t_k}+q_k\frac{\|A^T(A\theta_k-b)\|^2}{\frac12 (A\theta_k-b)^T (A\theta_k-b)}}.\]
Then for a very large $\delta t_k$ and $q_k\equiv\frac12$, we have 
\[\alpha_k\approx \frac{(A\theta_k-b)^T (A\theta_k-b)}{(A\theta_k-b)^T A A^T(A\theta_k-b)},\quad \beta_k=\frac{(A\theta_k-b)^TAA^T (A\theta_k-b)}{(A\theta_k-b)^T AA^T A A^T(A\theta_k-b)}.\]
Let $v_i$ be orthornormal eigenvectors of $A$ with eigenvalues of $\lambda_i$. 
Since $v_i$ form a basis for $\mathbbm R^N$,
let $A\theta_k-b=r=\sum_i r_i v_i$. Let $z=A^T(A\theta_k-b)$, then $z=A^T\sum_i r_i v_i=\sum_i r_i \lambda_i v_i$
and $Az=\sum_i r_i \lambda^2_i v_i$ .
We get
\[\alpha_k\approx \frac{r^T r}{ r^TA A^Tr}=\frac{r^T r}{ z^T z}=\frac{[\sum_i r_i v_i]^T[\sum_i r_i v_i]}{[\sum_i r_i \lambda_i v_i]^T[\sum_i r_i \lambda_i v_i]}=\frac{\sum_i r_i^2}{\sum_i \lambda_i^2 r_i^2}, 
\quad\beta_k=\frac{z^T z}{(Az)^T (Az)}=\frac{\sum_i  \lambda_i^2 r_i^2}{\sum_i \lambda_i^4 r_i^2} .\]
We can see that $\alpha_k$ is very similar to the optimal step size $\beta_k$ but not the same. 
In practice, a random initial guess $\theta_0$ usually makes $\alpha_k$ a descent step size in the first few or many iterations for $\delta t_k\equiv 1$
and $q_k\equiv q=\frac12$.

\section{Concluding remarks}\label{sec-remark}
 
We proposed in this paper a new  minimization algorithm inspired by the  scalar auxiliary variable (SAV) approach for gradient flows. Since the direct application of the SAV approach to  minimization problems may converge to wrong solutions,  we developed a  modified version of the SAV  approach  coupled with a relaxation  step and an adaptive stradegy. The new algorithm enjoys several distinct advantages, including  unconditionally energy  diminishing with a  modified energy, and  {empirical better performance than many first order methods}. In particular, it overcomes the difficulty in selecting proper step sizes  associated with the usual gradient based algorithms. The energy diminishing property  ensures the convergence, and the relaxation step, built on a connection between the decreasing modified energy and the original energy, {helps to accelerate the convergence}.

We also presented a converence analysis for {some SAV based algorithms} which include the new algorithm without the relaxation step as a special case.
Numerical results for several illustrative and benchmark problems indicates that the new algorithm is very robust and usually converges significantly faster than those popular gradient decent based methods.  

While we only considered a basic version of  the SAV based approach which already showed its promise, it is clear that this approach can be combined with other techniques of acceleration and generalization  to the gradient decent methods. How to further improve the robustness and accelerate the convergence rate of the SAV based approach will be the subject of a future study.

\section*{Acknowledgement}
{This work is partially supported by AFOSR FA9550-16-1-0102,  NSF DMS-2012585 and DMS-2208518.}
 
 \bibliographystyle{plain}
 \bibliography{reference}

\end{document}